\documentclass[english, reqno]{amsart}
\usepackage{amsfonts, amsmath, amssymb, amscd, amsthm, array, graphicx}
\usepackage[latin9]{inputenc}
\usepackage{amsmath,comment}
\usepackage{stmaryrd}
\usepackage{faktor} 
\usepackage{graphicx}
\usepackage{amssymb}
\usepackage{mathrsfs}
\usepackage{dsfont}
\usepackage{babel}
\usepackage{color}
\usepackage[all]{xy}
\usepackage[colorlinks=true, linkcolor=blue, citecolor=blue, urlcolor=blue, breaklinks=true]{hyperref}
\usepackage[capitalise]{cleveref}

\usepackage{mathtools}
\usepackage{hyperref,url}

\def\blfootnote{\xdef\@thefnmark{}\@footnotetext}
\makeatother
\usepackage{tikz}
\usetikzlibrary{arrows,shapes,snakes,automata,backgrounds,petri}
\usetikzlibrary{automata} 
\usetikzlibrary[automata] 
\usetikzlibrary{matrix,decorations.pathreplacing, decorations.markings}
\usepackage{geometry}
\usepackage{pdflscape}
\usepackage{graphicx}
\usetikzlibrary{external,automata,trees,positioning,shadows,arrows,shapes.geometric}

\usepackage{tikz-cd}
\usetikzlibrary{decorations.pathreplacing,decorations.markings}
\tikzset{close/.style={near start,outer sep=-2pt}} 
\tikzset{
  on each segment/.style={
    decorate,
    decoration={
      show path construction,
      moveto code={},
      lineto code={
        \path [#1]
        (\tikzinputsegmentfirst) -- (\tikzinputsegmentlast);
      },
      curveto code={
        \path [#1] (\tikzinputsegmentfirst)
        .. controls
        (\tikzinputsegmentsupporta) and (\tikzinputsegmentsupportb)
        ..
        (\tikzinputsegmentlast);
      },
      closepath code={
        \path [#1]
        (\tikzinputsegmentfirst) -- (\tikzinputsegmentlast);
      },
    },
  },
  mid arrow/.style={postaction={decorate,decoration={
        markings,
        mark=at position .5 with {\arrow[#1]{stealth}}
      }}},
}

\newlength\Textht
\newtheorem{thm}{Theorem}[section]
\newtheorem*{thm*}{Theorem}
\newtheorem{cor}[thm]{Corollary}
\newtheorem{lem}[thm]{Lemma}
\newtheorem{prop}[thm]{Proposition}

\theoremstyle{definition}
\newtheorem{defn}[thm]{Definition}

\newtheorem{obs}[thm]{Observation}
\newtheorem{fact}[thm]{Fact}

\newtheorem{rem}[thm]{Remark}
\newtheorem{exam}[thm]{Example}

\newtheorem{construction}[thm]{Construction}

\newcommand{\ZZ}{\mathbb{Z}}

\newcommand{\gh}{H_{\psi}}

\newcommand{\vect}{\mathbf}
\newcommand{\Hom}{\operatorname{Hom}}

\newcommand{\link}{\texttt{link}}
\newcommand{\st}{\texttt{star}}
\newcounter{mallikacomments}

\begin{document}

\title{Finitely presented kernels of right-angled Artin groups with abelian quotients.}

\author{Mallika Roy}
\address{Department of Mathematics, Indian Institute of Science (IISc), India.} \email{mallikaroy75@gmail.com}


\subjclass{20E05, 20E36, 20K15}

\keywords{right-angled Artin group, Bestvina--Brady group, finiteness property.}

\begin{abstract}
In this article, we characterise when the kernels of right-angled Artin groups with abelian quotients are finitely generated and finitely presented --- exhibiting an explicit finite generating set and finite presentation. As an application, we deduce the finite presentation of the kernel of any character of a right-angled Artin group. These results generalise the presentation of the kernel of a rational character of a right-angled Artin group given by Casals--Kazachkov--Roy and Dicks and Leary's presentations of Bestina--Brady groups. 
\end{abstract}

\maketitle

\section{Introduction}\label{intro}

Let $\Gamma$ be a finite simplicial graph. The associated right-angled Artin
group $G_\Gamma$ is the group having the presentation as follows. The set of generators of $G_\Gamma$ is the vertex set $V(\Gamma) = \{ a_1, a_2, \ldots a_s\}$, and $G_\Gamma$ has only the commuting relations stating that two generators commute if and only if their associated vertices are adjacent in $\Gamma$,
that is, the relators of $G_\Gamma$ are the commutators $[a_i,a_j]$ for each pair of adjacent vertices $(a_i, a_j)$ in $\Gamma$.
Let $\triangle_\Gamma$ be a non-empty finite flag complex with $1$-skeleton $\Gamma$ --- a finite simplicial complex that contains a simplex
bounding every complete subgraph of $\Gamma$.

Right-angled Artin groups contain many interesting subgroups, for example Bestvina--Brady groups. The Bestvina--Brady groups were first introduced in the seminal work of Bestvina and Brady~\cite{BB} as an answer to a long-standing open question regarding the existence of non-finitely presented groups of type $\mathbf{FP}$.
If one considers the kernels of the homomorphisms from $G_\Gamma$ to $\ZZ$ that map every standard generator of $G_\Gamma$ to $1$, then these kernels, denoted by $H_\Gamma$, are known as Bestvina--Brady groups in the literature.
Bestvina and Brady~\cite{BB} described the homological finiteness properties of the kernels $H_\Gamma$ in terms of the topology of the flag complex $\triangle_\Gamma$. In particular, they proved that $H_\Gamma$ is finitely presented if and only if $\triangle_\Gamma$ is simply connected and $H_\Gamma$ is of type $FP_{n+1}$ if and only if $\triangle_\Gamma$ is $n$-acyclic.

In this present article, we consider the kernel $H_\psi$ of an epimorphism $ \psi \colon G_\Gamma \to \ZZ^k$ and study the combinatorial presentation of the kernels when they are finitely presented.

In \cite{DL}, Dicks and Leary provided an explicit presentation for $H_\Gamma$. The generators in the presentation correspond to the edges $e_i$ of $\Gamma$ and in the case when $\triangle_\Gamma$ is simply connected, it is shown that the relations are of the form $e_1^{\epsilon}e_2^{\epsilon}e_3^{\epsilon}$ for each directed $3$-cycles $(e_1, e_2, e_3)$ of $\Gamma$ and $\epsilon = \pm 1$. This result gave an independent and purely algebraic proof that $H_\Gamma$ is finitely presented when $\triangle_\Gamma$ is simply connected.

The work of Bestvina and Brady was later extended by Meier, Meinert and VanWyk. Meier--Meinert--VanWyk~\cite{MMV} described the homotopical and homological $\Sigma$-invariants
of $G_\Gamma$ in terms of the topology of subcomplexes of $\triangle_\Gamma$ --- determining the finiteness properties of kernels of maps from $G_\Gamma$ to abelian groups.
In~\cite{MV} Meier--Vanwyk studied the Bieri--Neumann--Strebel invariant for $G_\Gamma$ in terms of the associated graph $\Gamma$ and gave an explicit description of the distribution of finitely generated normal subgroups of $G_\Gamma$ with abelian quotient.

Let $\varphi$ be a rational character of $G_\Gamma$, that is, $\varphi \colon G_\Gamma \twoheadrightarrow \ZZ$ is an epimorphism sending each generator $a_i$ of $G_\Gamma$ to an integer $n_i \in \ZZ$. Also, let $H_\varphi$ be the corresponding kernel. Recently, Casals--Kazachkov--Roy~\cite[Theorems 3.11, 3.16]{CKRoy} gave an algebraic characterisation of finitely generated and finitely presented kernels $H_\varphi$ generalising Dicks and Leary's presentations of Bestvina--Brady groups.

We show that if $G_\Gamma$ has a free abelian quotient with integral rank $k$ and $\triangle_\Gamma$ is $(k-1)$-$0$-connected (equivalently, $\Gamma$ is $k$-connected), then the kernel $H_\psi$ is finitely generated 
and we provide a finite generating set of $H_\psi$, see Theorem~\ref{thm:finite_generation}. Further, we extend our result and give an explicit finite presentation for $H_\psi$ when $\triangle_\Gamma$ is $(k-1)$-$1$-connected, see Theorem~\ref{thm:finite_presentation}. Moreover, in both cases we prove the converse as well, showing that how the combinatorial presentation of the kernel leads to the geometry of the underlying flag complex. Our results generalise~\cite{DL} and~\cite{CKRoy} recovering both the presentations of $H_\Gamma$ and $H_\varphi$ as it deal with finitely generated free abelian groups as codomains instead of infinite cyclic groups. 

The flag complex $\triangle_\Gamma$ is said to be $(k-1)$-$0$-connected, equivalently, $\Gamma$ is $k$-connected, if $\Gamma \setminus \{ v_1, \ldots, v_\ell\}$ is connected for any collection of $\ell < k$ vertices with $\ell \leqslant |V(\Gamma)|$. And the flag complex $\triangle_\Gamma$ is said to be $(k-1)$-$1$-connected if for any $\ell$ vertices $\{ v_1, \ldots, v_\ell\}$ with $0 \leqslant \ell \leqslant k-1$ and $\ell \leqslant |V(\Gamma)|$, the complex $\triangle_\Gamma \setminus \{  v_1, \ldots, v_\ell \}$ is $1$-connected, i.e., simply connected.

Let $\Gamma$ be a finite simplicial graph with $V(\Gamma) = \{a_1, \ldots, a_s \}$ and $\psi:G_\Gamma \to \ZZ^k$ be an arbitrary epimorphism sending $a_i$ to $\vect{n_i}$ for $i = 1, \dots, s$.

We use boldface font ($\vect{m}, \vect{n}, \dots$) to denote elements of the free-abelian group $\ZZ^k$ --- an arbitrary element of $\ZZ^k$ is $\vect{m} = (m_1, m_2, \ldots, m_k)$, where $m_i \in \ZZ$ for $i = 1, 2, \ldots, k$; and we denote $\widetilde{\vect{m}} = \max\{|m_1|, \dots, |m_k|\}$.
We consider a transverse set $\{ W_{\vect{m}} \, \, | \, \, \vect{m} \in \ZZ^k \}$ for the kernel $\gh \leqslant G_{\Gamma}$, where 
$\psi( W_{\vect{m}})=\vect{m}$ and $W_{\vect{0}}= 1$. The reader is referred to Section~\ref{sec: setting} for more details.

Our main results are the following:

\begin{thm*}[Finite generation, see Theorem \ref{thm:finite_generation}] Let $G_\Gamma$ be a RAAG and $\psi \colon G_\Gamma \rightarrow \ZZ^k$ be an epimorphism and $\ker \psi=\gh$. Then,  $\gh$ is finitely generated if and only if $\Gamma$ is $k$-connected. More precisely,
$$
\gh = \bigl\langle W_{\vect{m}}\, a_i \, W_{\vect{m} + \vect{n_i}}^{- 1} \, \mid \, 0 \leqslant \widetilde{\vect{m}} < kN^2R \bigr\rangle,
$$
where $N, \vect{n_i}, R$ are constants determined by $\psi$.
\end{thm*}
\begin{thm*}[Finite presentation, see Theorem \ref{thm:finite_presentation}]
The kernel $\ker \psi=\gh$ is finitely presented if and only if $\triangle_\Gamma$ is $(k-1)$-$1$-connected. Moreover, the the kernel $\gh$ has the following finite presentation
    \[
    \bigl\langle X_{\vect{m,i}} \, \mid \, R_1, R_2, R_3\bigr\rangle, 
    \]
    where $X_\vect{{m,i}}=W_{\vect{m}} \, a_i \, W_{\vect{m} + \vect{n_i}}^{-1}$, $0 \leqslant \widetilde{\vect{m}} < k N^2 R$, and the sets of relations $R_1, R_2, R_3$ are defined as follows:  
\begin{itemize}
\item[$R_1$:] for any directed $3$-cycle $(a_{i_1}, a_{i_2}, a_{i_3})$,
\[
 \tau(W_{\vect{s}}\,
  e_{i_1,i_2}^{Nd} e_{i_2,i_3}^{Nd}  e_{i_3,i_1}^{Nd}  \, W_{\vect{s}}^{-1}), \hspace{0.2cm} 0 \leqslant \widetilde{\vect{s}} < N, d=\pm 1, \text{ where } e_{i_r,i_p}=a_{i_r}a_{i_p}^{-1};
\]
\item[$R_2$:] $\tau(W_{\vect{s}}\, [a_i, a_j] \, W_{\vect{s}}^{-1})=1, \hspace{0.2cm} 0 \leqslant \widetilde{\vect{s}} < N$, where $(a_i, a_j) \in E(\Gamma);$
\item [$R_3$:] $X_{\vect{t,i}}=\tau(W_{\vect{t}}\, a_i (W_{\vect{t}+\vect{n_i}})^{-1})$ for $0\leqslant \widetilde{\vect{t}} < k^2N^3R_{\mathcal{M}}R$.
\end{itemize}
The map $\tau$ is a Reidemeister rewriting process, see Section~\ref{sec: reidemeister}, and $N, \vect{n_i}, R, R_{\mathcal{M}}$ are constants determined by $\psi$.
\end{thm*}

As an application, we deduce the finite presentation of the kernel a character $\chi \colon G_\Gamma \to \mathbb R$, see Corollary~\ref{cor:character}.

\section{Preliminaries and Notations}\label{prim-nt}
\subsection{Graph theory} We recall and set up some basic notation and terminologies in graph theory. 
Throughout this article, we assume finite graphs which have no loops and multi-edges, i.e., all the graphs are finite simplicial. 
Given a graph $\Gamma$, we denote the set of its vertices and edges by $V(\Gamma)$ and $E(\Gamma)$, respectively.
We denote $e=(a_i, a_j)$ to be an edge connecting vertices $a_i$ and $a_j$. $\iota(e), \tau(e)$ respectively denotes the initial point and the terminal point of the edge $e$, i.e., $e=(\iota(e), \tau(e))$. 
Two vertices are \emph{adjacent} if they are connected by an edge. 
Given any subset $V'$ of $V(\Gamma)$, the \textit{induced subgraph} (in some literature, it is also called as full subgraph) on $V'$ is a graph $\Gamma'$ whose vertex set is $V'$, and two vertices are adjacent in $\Gamma'$ if and only if they are adjacent in $\Gamma$.

A simplicial complex $L$ is called \emph{flag complex}  (also known as \emph{clique complex} in the literature) if every complete subgraph in the $1$-skeleton spans
a simplex in $L$. In other words, $L$ is completely determined by its $1$-skeleton.
Thus, the flag complex on $\Gamma$ is denoted by $\triangle_\Gamma$, contains an ($n$)-simplex bounding each complete induced subgraph with $(n+1)$-vertices of its 1-skeleton $\Gamma$.

We first recall that a graph $\Gamma$ is \emph{$m$-connected} if $\Gamma \setminus \{v_1, \ldots, v_{\ell} \}$ ($\ell < |V(\Gamma)|$) is connected for any collection of $\ell < m$ vertices.
In this article we use a notion of connectivity adopted from \cite{MMV} in the context of simplicial complexes.
A simplicial complex $L$ is \emph{$m$-$n$-connected} if for any $\ell$ vertices $\{ v_1, \ldots, v_\ell\}$ with $0 \leqslant \ell \leqslant m$ and $\ell < |V(L)|$, the complex $L \setminus \{ v_1, \ldots, v_\ell \}$ is $n$-connected. Hence, the notion of `$m$-connected' in graph theory is equivalent to the notion of $(m-1)$-$0$-connected.

Let $R$ be a commutative ring with unity. A complex is \emph{$n$-\textit{acyclic}} if its reduced homology groups (over $R$), up to and including dimension $n$, are trivial. Throughout this note we use the ring of integers $\ZZ$ for homological properties of the groups and their associated complexes.

\subsection{Homological finiteness properties}\label{finiteness_intro}

Finiteness properties $\mathbf{F_n}$ were introduced by C.T.C. Wall in~\cite{Wall}. A group $G$ is said to be of type $\mathbf{F_n}$ if it has an Eilenberg-MacLane complex $K(G,1)$ with finite $n$-skeleton. Equivalently, a group is of type $\mathbf{F_n}$ if it acts freely, faithfully, properly, cellularly, and cocompactly on an $(n-1)$-connected cell complex. A group is finitely generated if and only if it is of type $\mathbf{F_1}$, and is finitely presented if and only if it is of type $\mathbf{F_2}$. 

The finiteness properties $\mathbf{FP_n}$ were introduced by Bieri~\cite{Bieri}. The properties $\mathbf{FP_n}$ are weaker than the properties $\mathbf{F_n}$ --- if a group is of type $\mathbf F_n$, then it is of type $\mathbf{FP_n}$.
The group $G$ is said to be of type $\mathbf{FP_n}$ if there exists a resolution of the trivial $\mathbb ZG$--module $\mathbb Z$ such that the first $n$ terms are finitely generated projective $\mathbb Z G$--modules.  
One can easily deduce that the finitely presented groups are of type $\mathbf{FP_2}$ by using the chain complex of the universal cover of a presentation 2-complex. And the converse is not true as $\mathbf{FP_n}$ is weaker than $\mathbf{F_n}$.

The group $G$ is said to be of type $\mathbf{FP}$ if there exists a finite resolution
\[
0 \rightarrow P_n \rightarrow \cdots \rightarrow P_0 \rightarrow \mathbb Z \rightarrow 0
\]
of $\mathbb Z$ by finitely generated projectives over $\mathbb ZG$.

In~\cite{BB}, Bestvina--Brady provided an example of a group that is of type $\mathbf{FP}$ but is not finitely presented answering a long-standing open question.

\subsection{Bestvina--Brady groups}\label{bb_intro}

\begin{defn}
Let $\Gamma$ be a finite simplicial graph with the vertex set $V(\Gamma)$ and the edge set $E(\Gamma)$. 
The right-angled Artin group $G_{\Gamma}$ associated to $\Gamma$ has the following finite presentation:
\[
G_\Gamma = \bigl\langle V(\Gamma) \mid [a_i,a_j]=1 \text{ for each edge } (a_i,a_j) \in E(\Gamma) \bigr\rangle.
\]
Let $\varphi \colon G_\Gamma \rightarrow \ZZ$ be the group homomorphism sending all generators of $G_\Gamma$ to $1$. The \textit{Bestvina--Brady group} $H_\Gamma$ associated to $\Gamma$ is the kernel, $\ker \varphi$.
\end{defn}

We already mentioned in~\Cref{intro} that in~\cite{BB} Bestvina--Brady characterized the homological finiteness properties of $H_\Gamma$ in terms of the topology of the flag complex $\triangle_\Gamma$. More precisely,

\begin{thm}[\cite{BB}, Main Theorem]
Let $\Gamma$ be a finite simplicial graph.
\begin{itemize}
    \item[(1)] $H_\Gamma$ is finitely generated if and only if $\Gamma$ is connected.
    \item[(2)] $H_\Gamma$ is finitely presented if and only if $\triangle_\Gamma$ is simply-connected.
    \item [(3)] $H_\Gamma$ is of type $\mathbf{FP_{n+1}}$ if and only if $\triangle_\Gamma$ is $n$-acyclic.
\end{itemize}
\end{thm}

This result includes the J. Stallings' group~\cite{S} --- $H_\Gamma$ associated to the RAAG $F_2 \times F_2 \times F_2$ --- an example of finitely presented but not of type $\mathbf{FP_3}$ and R. Bieri's group~\cite{B} of type $\mathbf{FP_n}$ but not of type $\mathbf{FP_{n+1}}$, which is $H_\Gamma$ corresponding to the $\Gamma$, a join of $(n+1)$ pairs of points. 

The presentation of Bestvina--Brady groups was described by Dicks--Leary in~\cite{DL}:

\begin{thm}[\cite{DL}, Theorem 1, Corollary 3]
 Let $\triangle_\Gamma$ be connected. The group $H_\Gamma$ has a presentation with generators the set of directed edges of $\Gamma$, and relators all words of the form $e^n_1 e^n_2 \cdots e^n_{\ell}$, where $\ell,n \in \ZZ, n \geqslant 0,\ell \geqslant 2$, and $(e_1, \ldots,e_\ell)$ is a directed cycle in $\Gamma$.

 Moreover, if the flag complex $\triangle_\Gamma$ is simply connected. Then $H_\Gamma$ has presentation

\begin{equation*}
H_\Gamma = \bigl\langle e \in E(\Gamma) \mid ef = fe, ef = g \text{ if }\triangle(e, f, g) \text{ is a directed triangle } \bigr\rangle.    
\end{equation*}

In terms of the given generators for $G_\Gamma, e=\iota e(\tau e)^{-1} = a_i a_j^{-1}$ for every edge ${e} = (a_i, a_j)$ of $\Gamma$ (see Figure~\ref{directed triangle}).

\end{thm}

\begin{figure}[h]
    \centering
    \begin{tikzpicture}
    \tikzset{
    edge/.style={draw=black,postaction={on each segment={mid arrow=black}}}
} 
\node[fill=black!100, state, scale=0.10, vrtx/.style args = {#1/#2}{label=#1:#2}] (A) [vrtx=below/$a_i$]     at (0, 0) {};
\node[fill=black!100, state, scale=0.10, vrtx/.style args = {#1/#2}{label=#1:#2}] (C) [vrtx=above/$a_j$]    at (1, 1) {};
\node[fill=black!100, state, scale=0.10, vrtx/.style args = {#1/#2}{label=#1:#2}] (B) [vrtx=below/$a_k$]     at (2.5, 0) {};

\draw[edge] (A) -- (B) node[midway, below] {$g$};
\draw[edge] (C) -- (B) node[midway, above right] {$f$};
\draw[edge] (A) -- (C) node[midway, above left] {$e$};
    \end{tikzpicture}
    \caption{A directed triangle. 
    }
    \label{directed triangle}
\end{figure}
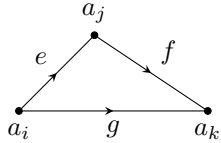


In the case where $\Gamma$ is connected and $\triangle_\Gamma$ is simply connected, the above presentation for $H_\Gamma$, called the \textit{Dicks--Leary presentation} \cite[Theorem 1, Corollary 3]{DL}. The Dicks--Leary presentation is not necessarily a minimal presentation. Dicks--Leary considered all the edges of $\Gamma$ as the generators. The simpler presentation was given by Papadima--Suciu in \cite{PS} proving that for the generators of $H_\Gamma$ it is sufficient to consider the edges of a spanning tree of $\Gamma$.

\begin{thm}[{\cite[Corollary 2.3]{PS}}]\label{Suciu-Papadima presentation}
If $\triangle_\Gamma$ is simply-connected, then $H_\Gamma$ has a presentation $H_\Gamma = F/R$, where $F$ is the free group generated by the edges of a spanning tree of $\Gamma$, and $R$ is a finitely generated normal subgroup of the commutator group $[F, F]$.
\end{thm}

The following example depicts a particular case when the Bestvina--Brady group $H_\Gamma$ is not isomorphic to any right-angled Artin group (see \cite[Proposition 9.4]{PS} for details).

\begin{exam}
Let $\Gamma$ be the graph in \Cref{fig:exam}. 
Choosing the spanning tree $T = \{e_1 , \ldots , e_5\}$ as indicated, the presentation of the Bestvina--Brady group is given as follows:
\[
H_\Gamma = \bigl\langle e_1,  \ldots, e_5 \mid [e_1 , e_2], [e_2 , e_3], [e_3 , e_4], [e_5 , {e_2}^{-1}e_3] \bigr\rangle.
\]
\begin{figure}[h]
    \centering
\begin{tikzpicture}[shorten >=1pt,node distance=20cm,auto]
\tikzset{
    edge/.style={draw=black,postaction={on each segment={mid arrow=black}}}
}
\node[fill=black!100, state, scale=0.10, vrtx/.style args = {#1/#2}{label=#1:#2}] (1) [vrtx=left/$a_6$] {};

\node[fill=black!100, state, scale=0.10, vrtx/.style args = {#1/#2}{label=#1:#2}] (2) [vrtx=right/$a_4$] [ below right of = 1] {};

\node[fill=black!100, state, scale=0.10, vrtx/.style args = {#1/#2}{label=#1:#2}] (3) [vrtx=left/$a_5$] [ below left of = 1] {};

\node[fill=black!100, state, scale=0.10, vrtx/.style args = {#1/#2}{label=#1:#2}] (4) [vrtx=right/$a_3$] [ below right of = 2] {};

\node[fill=black!100, state, scale=0.10, vrtx/.style args = {#1/#2}{label=#1:#2}] (5) [vrtx=below/$a_2$] [ below left of = 2] {};

\node[fill=black!100, state, scale=0.10, vrtx/.style args = {#1/#2}{label=#1:#2}] (6) [vrtx=left/$a_1$] [ below left of = 3] {};


\draw[edge] (6) -- (5) node[midway, below] {$e_1$};
\draw[edge] (5) -- (3) node[midway, left] {$e_2$};
\draw[edge] (5) -- (2) node[midway, right] {$e_3$};
\draw[edge] (5) -- (4) node[midway, below] {$e_4$};
\draw[edge] (2) -- (1) node[midway, right] {$e_5$};
\draw[edge] (6) -- (3);
\draw[edge] (3) -- (2);
\draw[edge] (3) -- (1);
\draw[edge] (4) -- (2);
\end{tikzpicture}
\caption{$\Gamma$ for which $H_\Gamma$ is not a RAAG.}
\label{fig:exam}
\end{figure}
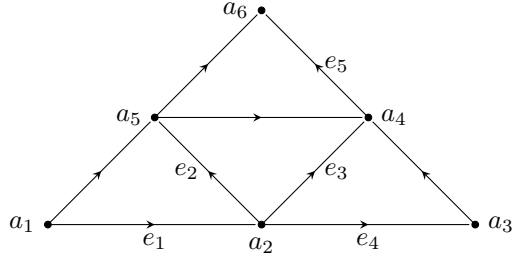
\end{exam}

\section{Main results}

Meier--Meinert--VanWyk~\cite{MMV} established the finiteness properties of kernels of maps from RAAGs to abelian groups. In particular Meier--Meinert--VanWyk~\cite{MMV} determined the $\mathbf{F_n}$ (and $\mathbf{FP_n}$) properties of the kernels of rational characters of RAAGs and the kernels $\gh$, where $\psi \colon G_\Gamma \twoheadrightarrow \ZZ^k$ is an epimorphism. In~\cite{MMV} the authors described the homotopical and homological $\Sigma$-invariants of $G_\Gamma$ in terms of the topology of subcomplexes of $\triangle_\Gamma$ (see also \cite{BG} for a different reinterpretation). We refer the reader to \cite{BNS} and \cite{BR} for background on the $\Sigma$-invariants. 

Casals--Kazachkov--Roy~\cite{CKRoy}  studied the kernel $H_\varphi$ of a rational character of a RAAG. Any rational character from a RAAG can be viewed as an epimorphism $\varphi \colon G_\Gamma \twoheadrightarrow \ZZ$ sending each generator $a_i$ of $G_\Gamma$ to an integer $n_i \in \ZZ$. 
Let $\Gamma_\varphi <\Gamma$ be the induced subgraph of $\Gamma$ spanned by the vertices that are mapped nontrivially by the rational character and let $\triangle_\varphi$ be the corresponding induced flag complex in $\triangle_\Gamma$.

In~\cite[Theorems 3.11, 3.16]{CKRoy} Casals--Kazachkov--Roy  generalised Dicks--Leary presentation of Bestvina--Brady groups and provide an algebraic proof for the result proven by Meier--Meinert--VanWyk~\cite[Corollary A]{MMV}.
More precisely, 
In~\cite[Theorem 3.11]{CKRoy} the authors proved that $H_\varphi$ is finitely generated if and only if the he subgraph $\Gamma_\varphi$ is connected and $0$-acyclic-dominating and gave an explicit finite generating set for $H_\varphi$. Casals--Kazachkov--Roy~\cite[Theorem 3.16]{CKRoy} also proved that $H_\varphi$ is finitely presented when $\triangle_\varphi$ is simply connected and $1$-acyclic dominating and produced an explicit finite presentation for $H_\varphi$. The authors employed the so-called Reidemeister--Schreier presentation to deduce the presentation of $H_\varphi \unlhd G_\Gamma$.

Note that $0$-\emph{acyclic-domination} requires that each vertex in $\triangle_\Gamma$ that is not in $\triangle_\varphi$ be adjacent to a vertex in $\triangle_\varphi$.
Similarly, $1$-\emph{acyclic-domination} requires that for each edge $(a_i,a_j)\in \triangle_\Gamma$, there is a vertex adjacent to $a_i$ and $a_j$ that belongs to $\triangle_\varphi$, i.e. either $a_i$ or $a_j$ are vertices in $\triangle_\varphi$ or there exists $a_k$ such that $(a_i, a_k), (a_j,a_k)$  are edges in $\triangle_\Gamma$ and $a_k$ is a vertex in $\triangle_\varphi$. The reader is referred to~\cite[Definition 1]{MMV} for the definition of an $n$-\emph{acyclic-dominating} subcomplex.

Motivated by the results of~\cite{CKRoy} and~\cite{MMV} and using the technique developed in~\cite{CKRoy}, in this note, we further generalise the result of~\cite{CKRoy} and we characterised when the kernel of an epimorphism $\psi \colon G_\Gamma \twoheadrightarrow \ZZ^k$ is finitely generated and finitely presented. Our theorems algebraically demonstrate the result of Meier--Meinert--VanWyk~\cite[Corollary B']{MMV} stated below.

\begin{thm}[\cite{MMV}, Corollary B']
    A right-angled Artin group $G_\Gamma$ has an abelian quotient of integral rank $k$ with $\mathbf{F_n}$ kernel if and only if $\triangle_\Gamma$ is $(k-1)$-$(n-1)$-connected.
\end{thm}

In this section, we give an algebraic proof of the fact that $\triangle_\Gamma$ is $(k-1)$-$0$-connected (equivalently, $\Gamma$ is $k$-connected) if and only if the kernel $H_\psi$ is finitely generated and we exhibit an explicit (possibly
infinite) presentation (see Theorem~\ref{thm:finite_generation}). Furthermore, if  $\triangle_\Gamma$ is $(k-1)$-$1$-connected, then we show that the $H_\psi$ is finitely presented and we exhibit an explicit presentation (see Theorem~\ref{thm:finite_presentation}).

We remind that $V(\Gamma) = \{ a_1, \ldots, a_s\}$ and $\Gamma$ is $k$-connected if $\Gamma \setminus \{ a_1, \ldots a_\ell\}$ is connected for any collection of $\ell < k$ vertices and $\ell < s$. The flag complex $\triangle_\Gamma$ is $(k-1)$-$1$-connected if for any $\ell$ vertices $\{ a_1, \ldots, a_\ell\}$ with $0 \leqslant \ell \leqslant k-1$ and $\ell \leqslant s$, the complex $\triangle_\Gamma \setminus \{  a_1, \ldots, a_\ell \}$ is $1$-connected, that is, simply connected.

We record the following fact from the proof of \cite[Theorem 6.3]{MV} which explains that if $\Gamma$ is $k$-connected, then the living subgraph $\Gamma_{\zeta \circ \psi}$ is connected for any non-zero $\zeta \in \Hom(\ZZ^k, \ZZ)$. And we further extend this to focus the interplay between $(k-1)$-$1$-connected complex $\triangle_\Gamma$ and its living subcomplex $\triangle_{\zeta \circ \psi}$.

\begin{fact}\label{fact 1}
    For any integer $d$ there always exist a set of $d$ elements of $\ZZ^k$ with the property that any $k$ of them are linearly independent. $G_\Gamma$ has $s$ many generators namely $a_1, \ldots, a_s$. Let $X = \{\vect{n_1}, \ldots, \vect{n_s} \}$ be such a set and recall that $\psi(a_i) = \vect{n_i}$ for each $i = 1, \ldots, s$. Notice that as $\zeta \in \Hom(\ZZ^k, \ZZ)$ is non-zero, it maps at most $k-1$ elements of $X$ to $0$. Thus it follows that $\Gamma_{\zeta \circ \psi}$ is connected since $\Gamma$ is $k$-connected. Also, for each vertex $a_t$ in $\Gamma$ that is not in $\Gamma_{\zeta \circ \psi}$ be adjacent to a vertex in $\Gamma_{\zeta \circ \psi}$, otherwise $\link(a_t)$ disconnects $\Gamma$ as $|V(\link(a_t))| \leqslant (k-2)$. In other words, $\Gamma_{\zeta \circ \psi}$ is $0$-acyclic-dominating.

    Now we use the similar argument to show that if $\triangle_\Gamma$ is $(k-1)$-$1$-connected then living subcomplex $\triangle_{\zeta \circ \psi}$ is $1$-connected and $1$-acyclic dominating for any non-zero $\zeta \in \Hom(\ZZ^k, \ZZ)$. Firstly, as the complement of $\triangle_{\zeta \circ \psi}$ contains at most $(k-1)$ vertices, hence, if $\triangle_{\zeta \circ \psi}$ is not $1$-connected, then $\triangle_\Gamma$ is not  $(k-1)$-$1$-connected, which is a contradiction. Moreover, for each edge $(a_i, a_j) \in E(\triangle_\Gamma)$ there is a vertex in $\triangle_{\zeta \circ \psi}$ lie at distance $1$ from both $a_i$ and $a_j$, i.e. either $a_i$ or $a_j$ are vertices in $\triangle_{\zeta \circ \psi}$ or there exists $a_k$ such that $(a_i, a_k), (a_j,a_k)$  are edges in $\triangle_\Gamma$ and $a_k$ is a vertex in $\triangle_{\zeta \circ \psi}$ and so $\triangle_{\zeta \circ \psi}$ is $1$-acyclic-dominating.
\end{fact}  

Using the above Fact~\ref{fact 1}, we construct an induced subgraph $\Gamma_\chi$ of $\Gamma$ which we use throughout the note, particularly for proving the Theorems~\ref{thm:finite_generation} and~\ref{thm:finite_presentation}.

\begin{construction}\label{cons 1}
    We can consider a non-zero $\zeta^{(j)} \in \Hom(\ZZ^k, \ZZ)$ as the $j$-th coordinate projection map, that is for any $\vect{m} \in \ZZ^k$, $\zeta^{(j)}(\vect{m}) = m_j$. 
    Hence, employing Fact~\ref{fact 1} for each $j \in \{ 1, \dots, k\}$ $\Gamma_{\zeta^{(j)} \circ \psi}$ is connected and $0$-acyclic dominating (resp. $\Gamma_{\zeta^{(j)} \circ \psi}$ is $1$-connected and $1$-acyclic dominating) if $\Gamma$ is $k$-connected (resp. $\triangle_\Gamma$ is $(k-1)$-$1$-connected) and $\psi \colon G_\Gamma \to \ZZ^k$ is an epimorphism.
    Let us denote $\Gamma_\chi = \Gamma_{\zeta^{(1)} \circ \psi} \cap \cdots \cap \Gamma_{\zeta^{(k)} \circ \psi}$.

     We claim that $\Gamma_\chi$ is connected and $0$-acyclic-dominating. $\Gamma_\chi \leqslant \Gamma$ is connected follows directly as $\Gamma$ is $k$-connected and in particular, $\Gamma$ is connected. Let $a_i \in V(\Gamma)$ be any vertex. Since $\Gamma_{\zeta^{(j)} \circ \psi}$ is $0$-acyclic-dominating, there exists $a^{(j)}_\ell \in V(\Gamma_{\zeta^{(j)} \circ \psi})$ such that either $a_i = a^{(j)}_\ell$ or $(a_i, a^{(j)}_\ell) \in E(\Gamma)$. And this holds for each $j \in \{ 1, \dots, k\}$.
     Equivalently, $a_i \in \link(\Gamma_\chi)$ for any $a_i \in V(\Gamma)$. This proves our claim that $\Gamma_\chi$ is $0$-acyclic-dominating.

     Similarly, one can prove that $\triangle_{\Gamma_\chi}$ is $1$-connected and $1$-acyclic-dominating when $\triangle_\Gamma$ is $(k-1)$-$1$-connected.


     Also note that if $a_i \in V(\Gamma)$ belongs to $\Gamma_\chi$, then $n^j_i \neq 0$ for each $j = 1, \dots, k$ and hence $\vect{n_i} \neq \vect{0}$, that is, $a_i \in \Gamma_\psi$.
    \end{construction}

We state the following Proposition combining~\cite[Theorem 1.1]{MV} and Fact~\ref{fact 1}. 

\begin{prop}\label{prop 1}
    An epimorphism $\psi \colon G_\Gamma \to \ZZ^k$ has a finitely presented kernel if and only if for every non-zero $\zeta \in \Hom(\ZZ^k, \ZZ)$, $\triangle_{\zeta \circ \psi}$ is simply connected and $1$-acyclic-dominating.
\end{prop}

\begin{proof}
  The proof of the Proposition is straightforward and analogous to~\cite[Proposition 6.2]{MV}.  
\end{proof}


\subsection{Setting.}\label{sec: setting} We are considering the epimorphism $\psi \colon G_{\Gamma} \twoheadrightarrow \ZZ^k$ and its corresponding kernel. Since, $\psi$ is surjective, 
there exists $W_{\vect{m}} \in G_{\Gamma}$ such that $\psi(W_{\vect{m}})=\vect{m}$, where $W_{\vect{m}}$ is defined in the following way. As $V(\Gamma)=\{\vect{a_1},\dots, \vect{a_s}\}$ and $\psi(a_i) = \vect{n_i}$, without loss of generality (and up to relabelling the vertices), we can take  $W_{\vect{m}} = a_1^{R_1(\vect{m})} \cdots a_k^{R_k(\vect{m})}$, where $R_i(\vect{m}) \in \ZZ$ for any $\vect{m} \in \ZZ^k$, $i =1, \ldots, k$ and $\vect{n_i} \ne \vect{0}$, in particular, $n^j_i \neq 0$ for each $i, j \in \{1, \dots, k\}$. 
In this setting, the vertices $a_1, \ldots, a_k$ present both in $\Gamma_{\psi}$ and $\Gamma_{\chi}$.

Also, if $\vect{m} \ne \vect{0}$, then $R_i(\vect{m}) \ne 0$ for $i=1, \dots, k$. For any integer $t \in \ZZ$, $W_{t\vect{m}} = a_1^{tR_1(\vect{m})} \cdots a_k^{tR_k(\vect{m})}$. Note that $W_{-\vect{m}} = a_1^{R_1(\vect{-m})} \cdots a_k^{R_k(\vect{-m})}$ and $(W_{\vect{m}})^{-1} = \left( a_1^{R_1(\vect{m})} \cdots a_k^{R_k(\vect{m})} \right)^{-1} = \left( a_k^{-R_k(\vect{m})} \cdots a_1^{-R_1(\vect{m})}  \right) \neq W_{-\vect{m}}$.

From $\psi(W_{\vect{m}})= \vect{m}$, we have  $\vect{n_{1}}{R_1(\vect{m})} + \cdots + \vect{n_k}{R_k(\vect{m})} = \vect{m}$. Set $R_i(\vect{0}) = 0$ for $i=1, \dots, k$. Hence, $W_{\vect{0}} = 1$ in $G_\Gamma$.

\begin{rem}\label{rem: on R} For any two elements $\vect{m}, \vect{n} \in \ZZ^k$, we claim that $R_j(\vect{m} + \vect{n}) = R_j(\vect{m}) + R_j(\vect{n})$ for all $j = 1, \ldots k$.
Following our notation $W_{\vect{m} + \vect{n}} = a_1^{R_1(\vect{m} + \vect{n})} \cdots a_k^{R_k(\vect{m} + \vect{n})}$ and $\psi(W_{\vect{m} + \vect{n}}) = \vect{m} + \vect{n}$,
\begin{equation}\label{eq: compare 1}
    \vect{n_1}R_1(\vect{m} + \vect{n}) + \cdots + \vect{n_k}R_k(\vect{m} + \vect{n}) = \vect{m} + \vect{n}.
\end{equation}
Similarly, as $\psi(W_{\vect{m}}) = \vect{m}$ and $\psi(W_{\vect{n}}) = \vect{n}$, we get that
$
\vect{n_1}R_1(\vect{m}) + \cdots + \vect{n_k}R_k(\vect{m}) = \vect{m},
$
and
$
\vect{n_1}R_1(\vect{n}) + \cdots + \vect{n_k}R_k(\vect{n}) = \vect{n}.
$
Adding these, we have
\begin{equation}\label{eq: compare 2}
   \vect{n_1} \left(R_1(\vect{m}) + R_1(\vect{n}) \right) + \cdots + \vect{n_k}\left(R_k(\vect{m}) + R_k(\vect{n}) \right) = \vect{m} + \vect{n}. 
\end{equation}
Now comparing equations~\eqref{eq: compare 1} and ~\eqref{eq: compare 2} it follows that $R_j(\vect{m} + \vect{n}) = R_j(\vect{m}) + R_j(\vect{n})$ for all $j = 1, \ldots k$.
\end{rem}

Let $\mathcal{M} \subset \ZZ^k$ be the set of all elements $\vect{v}$ such that $\widetilde{\vect{v}} \leqslant \widetilde{\vect{m}}$, that is, $\mathcal{M} = \{ \vect{v} \in \ZZ^k \mid \widetilde{\vect{v}} \leqslant \widetilde{\vect{m}} \}$. Let $R_\mathcal{M} = \underset{j}{\max}\,\{ R_j(\vect{v}) \mid \vect{v} \in \mathcal{M}\}$.

We set up the following notation to deal with the elements of $G_\Gamma$ with exponents as elements of $\ZZ^k$.

\begin{itemize}

\item For any two elements $\vect{m} = (m_1, \dots, m_k), \vect{n} = (n_1, \dots, n_k) \in \ZZ^k$, we have the following: $\vect{m} + \vect{n} = (m_1 + n_1, \dots, m_k + n_k)$;  $\vect{m} \vect{n} = (m_1 n_1, \dots, m_k  n_k)$;  $\vect{m} / \vect{n} = (m_1 /n_1, \dots, m_k / n_k)$.

\vspace{0.1 cm}

\item We denote $(1, 1, \dots, 1)$ as $\vect{1}$ and $[\vect{t}] = (t, t, \dots, t)$ for any integer $t \in \ZZ$.

\vspace{0.1 cm}

\item If $\vect{m} = (m_1, \ldots, m_k) \in \ZZ^k$, then we denote the element $\vect{m^j} = (0, \ldots, m_j, \ldots, 0)$, i.e., $m^j_i = 0$, for $i = 1, \ldots, k$ and $i \neq j$. Similarly, we denote $\vect{n^j_i} = (0, \ldots, n_{i_j}, \ldots, 0)$ for $i = 1, \ldots, s$ and $j = 1, \ldots, k$.

\vspace{0.1 cm}

\item For any $a_i \in V(G_\Gamma)$, 
$a_i^{\vect{m}} = W_{m_1\vect{n^1_i}} \cdots W_{m_k\vect{n^k_i}}$ and $a_i^{-\vect{m}} =(a_i^{\vect{m}})^{-1}$; more generally $a_i^{t\vect{m}} = (a_i^{\vect{m}})^t$, for any integer $t \in \ZZ$. Note that $\psi(W_{m_j\vect{n^j_i}}) = m_j \vect{n^j_i}$ for each $j = 1, \ldots, k$. Hence $\psi(a_i^{\vect{m}}) = m_1 \vect{n^1_i} + \cdots + m_k \vect{n^k_i} = m_1(n_{i_1}, 0,\ldots, 0) + \cdots + m_k(0, \ldots, 0, n_{i_k}) = (m_1n_{i_1}, \ldots, m_kn_{i_k}) = \vect{m} \vect{n_i}$.

In a similar way, for an arbitrary element $X \in G_\Gamma$ which is mapped to $\vect{n_x}$ (say), we set $X^{\vect{m}} = W_{m_1\vect{n^1_x}} \cdots W_{m_k\vect{n^k_x}}$ and so $\psi(X^{\vect{m}}) = \vect{m}\vect{n_x}$, so $X^{-\vect{m}} = (X^{\vect{m}})^{-1}$. 

Since commuting is the only relation in $G_\Gamma$, suppose that $[a_i, a_j] = 1$ in $G_\Gamma$. Then for any $t \in \ZZ$ and $\vect{d} ,\vect{m_i}, \vect{m_j} \in \ZZ^k$ the followings hold:
    $(a_i^{\vect{m_i}} \, a_j^{\vect{m_j}}) = (a_j^{\vect{m_j}} \, a_i^{\vect{m_i}})$, 
    and so  
    $a_i^{t\vect{m_i}} \, a_j^{t\vect{m_j}}  = (a_i^{\vect{m_i}} \, a_j^{\vect{m_j}})^t$.
Also, for any two commuting generators $a_i, a_j$ in $G_\Gamma$ we set that
$(a_i^{\vect{d}\vect{m_i}} \, a_j^{\vect{d}\vect{m_j}}) = (a_i^{\vect{m_i}} \, a_j^{\vect{m_j}})^\vect{d}$. On the other hand we set $a_i^t \, a_i^{\vect{m}} = a_i^t \, a_i^{\vect{m}/\vect{1}} = a_i^{\frac{[\vect{t}] + \vect{m}}{\vect{1}}} = a_i^{[\vect{t}] + \vect{m}}$.
\end{itemize}


Our goal is to describe a presentation for the kernel $\ker \psi$, which we denote by $\gh$ in the case when $G_\Gamma \simeq \gh \rtimes \ZZ^k$ and $\triangle_\Gamma$ is $(k-1)$-$1$-connected. Let us take $\{ W_\vect{m} \, \, | \, \, \vect{m} \in \ZZ^k\}$ as the Reidemeister set for the subgroup $\gh \unlhd G_\Gamma$. Then Reidemestier--Schreier allows us to have a (possibly infinite) presentation for $\gh$.  In this section our goal is to produce  a \emph{finite} presentation (derived from the Reidemestier--Schreier presentation) for $\gh$
using the underlying geometry of $\triangle_\Gamma$.

\subsection{The Reidemeister-Schreier presentation.}\label{sec: reidemeister}
The Reidemeister-Schreier method is a technique for producing presentations (in general infinite) of a subgroup $H$ of a group $G$ from the presentation of $G$. There are many variants of this method, the most common being the one suggested by Schreier, which chooses a set of transverse elements with the extra property that it is closed under subwords, to obtain a simpler presentation. In our case, we will use the variant described by Reidemeister, see for instance \cite[Theorem 2.8]{MKS}, which allows any choice of a transverse set at the price of some extra relations. We recall this version in the following theorem:

\begin{thm}[Reidemeister-Schreier, see Theorem 2.8 in \cite{MKS}] Let $G = \bigl\langle s \in S \mid r \in  R\bigr\rangle$ be a group. Let $H \unlhd G$ be a normal subgroup, let $T$ be a set of right coset representatives of $H$ in $G$, and let $\overline{\cdot} \colon G \to T$, $w \to \overline{w}$ be a right coset representative function. Then, $H$ has a presentation
$$
\bigl\langle x_{s,t} \mid x_{s,t}=\tau(ts (\overline{ts})^{-1}), \tau(t r t^{-1})=1,  s\in S, t\in T \bigr\rangle, 
$$
under the mapping $x_{s,t}\to ts (\overline{ts})^{-1}$ and where $\tau$ is the Reidemeister rewriting process for words in $H$.
\end{thm}

Let $H \unlhd G=\bigl\langle a_1, \dots, a_s\bigr\rangle$ and let $w\to \overline{w}$ be a right coset representative function for $G$ mod $H$. Let $u = a_{j_1}^{\epsilon_1} a_{j_2}^{\epsilon_2}\cdots a_{j_r}^{\epsilon_r}\in H$, where $\epsilon_i = \pm 1$.

The rewriting process $\tau$ expresses the word $u$ as a product of the generators $\overline{w} \, a_i \, \overline{wa_i}^{-1}$. More precisely, 
$$
\tau(u) = (\overline{w_1}\, a_{j_1}^{\epsilon_1} \, \overline{w_1a_{j_1}^{\epsilon_1}}^{-1}) (\overline{w_2} \,a_{j_2}^{\epsilon_2} \, \overline{w_2a_{j_2}^{\epsilon_2}}^{-1}) \dots (\overline{w_r}\, a_{j_r}^{\epsilon_r} \, \overline{w_r a_{j_r}^{\epsilon_r}}^{-1}),
$$
where
$$
w_1=1, \, w_2=w_1a_{i_1}^{\epsilon_1}, \, \dots, \, w_r=w_{r-1}a_{i_{r-1}}^{\epsilon_{r-1}}.
$$

The notation used for the Reidemeister rewriting process $\tau$ is followed from~\cite{CKRoy}. The reader is referred to~\cite[Definition 3.4]{CKRoy} for more details.



\textbf{Choice of right coset representative.} We next choose the transverse elements for the subgroup $\gh$ as follows. We define $W_{\vect{m}}$ to be the element $a_{1}^{R_1(\vect{m})} \cdots a_{k}^{R_k\vect{m})}$ of $G_\Gamma$, where $\vect{m}\in \ZZ^k$, $\vect{m}\ne \vect{0}$; for $\vect{m}=\vect{0}$, we define $W_{\vect{0}}=1$. From the definition, we have that $\psi(W_{\vect{m}})= R_1( \vect{m}) \vect{n_1} + \cdots+ R_k(\vect{m}) \vect{n_k} = \vect{m}$. Let us take $\{ W_{\vect{m}} \, \, | \, \,  \vect{m} \in \ZZ^k\}$ as transverse set for the normal subgroup $\gh \leqslant G_\Gamma$. 

From the Reidemeister--Schreier Theorem, we have that the kernel $\gh$ admits the following presentation:

\begin{thm}[Reidemeister--Schreier]\label{thm: R-S presentation}
In the above notation, the presentation of 
 $\gh \unlhd G_\Gamma$ is,
 $$
 \gh = \bigl\langle X_{\vect{m,i}} \, \, \mid \, \, X_{\vect{m,i}} = \tau(W_{\vect{m}}\,  a_i \, (W_{\vect{m+n_i}})^{-1}), \tau(W_{\vect{m}}\, [a_i,a_j]\,  W_{\vect{m}}^{-1}) = 1 \bigr\rangle,
 $$
 where $X_{\vect{m,i}} = W_{\vect{m}} \, a_i \, (W_{\vect{m+n_i}})^{-1}$, $\vect{m} \in \ZZ^k$, $a_i\in V(\Gamma)$, and $(a_i, a_j) \in E(\Gamma)$.
\end{thm}


\begin{rem}\label{rem:Reidemeister_process}
In our specific case, the Reidemeister rewriting process for the word $u=a_{j_1}^{\epsilon_1} a_{j_2}^{\epsilon_2}\cdots a_{j_r}^{\epsilon_r}$ in $H_\psi$ takes the form:
$$
\tau(u)=X_{\vect{m_1,j_1}}^{\epsilon_1}\, X_{\vect{m_2, j_2}}^{\epsilon_2} \dots \, X_{\vect{m_r,j_r}}^{\epsilon_r},
$$
where $m_l$ is either $\psi(a_{i_1}^{\epsilon_1}\dots a_{i_{l-1}}^{\epsilon_{l-1}})$ if $\epsilon_l=1$, or $\psi(a_{i_1}^{\epsilon_1}\dots a_{i_{l}}^{\epsilon_{l}})$ if $\epsilon_l=-1$.

For instance, suppose that $a_1 a_2^{-1} a_3\in \gh$, i.e. $\vect{n_1}-\vect{n_2} + \vect{n_3} = \vect{0}$. Then 
$$
\tau(a_1a_2^{-1}a_3)=(a_1 \, W_{\vect{n_1}}^{-1})\, (W_{\vect{n_1}} \, a_2^{-1} \, W_{\vect{n_1}-\vect{n_2}}^{-1})\, (W_{\vect{n_1}-\vect{n_2}}^{-1}\,  a_3 \, W_{\vect{n_1}-\vect{n_2}+ \vect{n_3}}^{-1})= X_{\vect{0,1}} \,X_{\vect{n_1-n_2, 2}}^{-1} \, X_{\vect{n_1-n_2, 3}}.
$$

\end{rem}

Note that the Reidemeister--Schreier presentation is independent of the geometry of $\triangle_\Gamma$. We record the following Lemma from~\cite{CKRoy}.
\begin{lem} \label{lem:trasverse_homomorphism}
The map $\tau$ is independent of the choice of the word $\omega$ in the free group, i.e. if $\omega$ and $\omega'$ define the same element in the free group then $\tau(\omega)=\tau(\omega')$.
\end{lem}
The reader is referred to~\cite[Lemma 3.7]{CKRoy} for the proof. 



We introduce some notation using the connectivity of the graph $\Gamma$ which will be used in the remaining part of this note.

\subsection{Notation} \label{sec: notation}
We assume that $\Gamma$ is $k$-connected, it follows from Fact~\ref{fact 1} that $\Gamma_{\chi}$ is connected where $\Gamma_{\chi}$ is defined as in Construction~\ref{cons 1}.
\begin{itemize}

    \item (Paths) For any two vertices $a_i\neq a_j$ in $V(\Gamma)$, a path $p_{i,j}$ from $a_i$ to $a_j$ in $\Gamma$ is a sequence of vertices $a_i=a_{i_1}, a_{i_2}, \dots, a_{i_{d+1}}=a_j$ such that $a_{i_k}\in V(\Gamma)$ and $(a_{i_{k-1}}, a_{i_k})\in E(\Gamma)$, for $k=2, \dots, d+1$. The length of the path $d$ is the number of vertices minus $1$. If the path has length $0$, i.e. it is a vertex, say $a_i$, we denoted $q_{i,i}$ simply by $a_i$. If the path has length $1$, namely is given by two adjacent vertices $a_i, a_j$, we denote it by $e_{i,j}$. Thus a path of length more than $0$ can also be thought as a sequence of edges $e_{i_1, i_2}, e_{i_2, i_3}, \dots, e_{i_{d}, i_{d+1}}$.

    Notice that a path can be defined for any two vertices in $\Gamma$, as $\Gamma$ is $k$-connected and in particular, $\Gamma$ is connected.


    \vspace{0.1 cm}

    \item (Least common multiple of subsets of vertices) For a single vertex $a_i$ which maps to $\vect{n_i}$, let $N_i$ be the least common multiple (lcm) of $|n_{i_1}|, \ldots, |n_{i_k}|$. Denote by $N$ lcm of $\{ \vect{n_i} \mid a_i \in V(\Gamma_\chi)\}$  and  clearly, $N \neq 0$.
    For any two adjacent $(a_i,a_j)\in E(\Gamma_\chi)$, let $N_{ij}$ be the lcm of $\vect{n_i}$ and $\vect{n_j}$.
     Since, $[\vect{N}] = (N, \ldots, N)$, then from the construction $[\vect{N}]$ is divisible by each $\vect{n_i}$, $i = 1, \dots, k$. Later we use this to define \emph{vectorised weighted edge} and \emph{vectrorised weighted path}.
    Consider a path $p$ in $\Gamma_\chi$, $p=a_{i_1}, \dots, a_{i_k}$. Then we define $N(p)$ to be the lcm of $N_{i_1}, \dots, N_{i_k}$. By definition, $N_{ij}, N(p)$ divide $N$ for any $(a_i, a_j)\in E(\Gamma_\chi)$ and any path $p$ in $\Gamma_\chi$. 
    

    \vspace{0.1 cm}



    \item (Vectorised weighted paths in $\Gamma_\chi$) For any $\vect{m}$ which is divisible by $[\vect{N_{ij}}]$, and $(a_i, a_j)\in E(\Gamma_\chi)$, we define the \emph{vectorised weighted edge} $e_{i,j}^\vect{m}$ as $(a_i^{\vect{m}/\vect{n_i}} \, a_j^{-\vect{m}/\vect{n_j}})$. Similarly, if $\vect{m}$ divisible by $[\vect{N}(p)]$, for any path $p = p_{i,j}$ in $\Gamma_\chi$ between two distinct vertices $\vect{a_i}$ and 
     $\vect{a_j}$ (not necessarily adjacent) 
%
    we define the \emph{vectorised weighted path} $p_{i,j}^\vect{m} = (a_i^{\vect{m}/\vect{n_i}} a_j^{-\vect{m}/\vect{n_i}})$.
    Thus, one can think the \emph{vectorised weighted path} $p_{i,j}^\vect{m}$
    as the product of the vectorised weighted edges that define the path $p_{i,j}$, i.e., if $p_{i,j} = e_{i_1, i_2} \, e_{i_2, i_3} \dots e_{i_d, i_{d+1}}$, then $p_{i,j}^{\vect{m}} = {e}_{i_1, i_2}^\vect{m} \, {e}_{i_2, i_3}^\vect{m} \dots {e}_{i_d, i_{d+1}}^\vect{m}$. Notice that by definition, since $\vect{m}$ is divisible by $[\vect{N}(p)]$, it is also divisible by $[\vect{N_{i_k, i_{k+1}}}]$ so the vectorised weighted edges are well-defined.

    For both vectorised weighted edge and vectorised weighted path, we use the following convention. If $\vect{m} = (m, \dots, m) $, i.e., if each $m_i = m$ for $i = 1, \dots, k$, and $m$ is divisible by $N_{ij}$ then we simply write $e^m_{i,j}$, instead of $e^{[\vect{m}]}_{i,j}$ for any $(a_i, a_j) \in E(\Gamma_\chi)$. 
    Similarly for any two vertices $a_j, a_k \in V(\Gamma_\chi)$, we write $p^m_{j, k}$, instead of $p^{[\vect{m}]}_{j, k}$ if $m$ is divisible by $N(p)$ and $[\vect{m}] = (m, \dots, m)$.

    Observe that in $G_\Gamma$, one has ${e}_{i,j}^{\vect{m}} = (a_i^{\vect{m}/\vect{n_i}} a_j^{-\vect{m}/\vect{n_j}}) = ({e}_{j,i}^{\vect{m}})^{-1}$ and ${p}_{i,j}^{\vect{m}} = ({p}_{j,i}^{\vect{m}})^{-1}$. On the other hand, while ${e}_{i,j}^{\vect{m}} = {e}_{j,i}^{\vect{-m}}$, we have that  ${p}_{i,j}^{\vect{m}} \neq {p}_{j,i}^{\vect{-m}}$ whenever $[a_i, a_j]\neq 1$ in $G_\Gamma$. 



\end{itemize}

We denote $R_{\mathcal{N}}$ simply by $R$, where $ \mathcal{N} =\{ \vect{v} \in \ZZ^k \mid \widetilde{\vect{v}} \leqslant N\}$. Notice that $R = \underset{j}{\max}\,\{ R_j(\vect{v}) \mid \vect{v} \in \mathcal{N}\} >0$ since by assumption $\vect{n_j} \ne \vect{0}$ for $j=1, \dots, k$ and so $R_j(\vect{n_j}) \neq 0$ for $j=1, \dots, k$.


\begin{lem}\label{lem:edges}

    For each adjacent pair of vertices $a_j, a_k\in V(\Gamma_\chi)$ we have that 
    $e_{j,k}^{N_{jk}} \in \bigl\langle X_{\vect{m, i}} \mid 0\leqslant \widetilde{\vect{m}} < kN^2R \bigr\rangle.$
    In particular, $\bigl\langle e_{j,k}^{N_{jk}} \mid (a_j, a_k)\in E(\Gamma_\chi)\bigr\rangle \leqslant \bigl\langle X_{\vect{m,i}} \mid a_i \in V(\Gamma_\chi), 0 \leqslant \widetilde{\vect{m}} < kN^2R \bigr\rangle$.
\end{lem}  

\begin{proof}
We first write  
$$
e_{j,k}^{N_{jk}}=(a_j^{[\vect{N_{jk}}]/\vect{n_j}} \, a_k^{-[\vect{N_{jk}}]/\vect{n_k}}) =  (a_j^{[\vect{N_{jk}}]/\vect{n_j}} \,W_{[\vect{N_{jk}}]}^{-1})\,(a_k^{[\vect{N_{jk}]}/\vect{n_k}} \,W_{[\vect{N_{jk}}]}^{-1})^{-1}.
$$
Let us denote $\vect{d_j} = [\vect{N_{jk}}]/\vect{n_j}$ and $\vect{d_k} = [\vect{N_{jk}}]/\vect{n_k}$. Now we show from the direct computation that $(a_j^{[\vect{N_{jk}}]/\vect{n_j}} \, W_{[\vect{N_{jk}}]}^{-1}) \in \bigl\langle X_{\vect{m,i}} \mid a_i \in V(\Gamma_\chi),\, 0 \leqslant \widetilde{\vect{m}} < kN^2R \bigr\rangle$.
$$
\begin{array}{l}
(a_j^{[\vect{N_{jk}}]/\vect{n_j}} \, W_{[\vect{N_{jk}}]}^{-1})      \\
\vspace{0.1 cm}
= (a_j^{\vect{d_j}} \, W_{[\vect{N_{jk}}]}^{-1})     \\
\vspace{0.1 cm}
= W_{d_{j_1}\vect{n^1_j}} \cdots W_{d_{j_k}\vect{n^k_j}} \, W_{[\vect{N_{jk}}]}^{-1}\\
\vspace{0.1 cm}
= \left(a_1^{d_{j_1}R_1(\vect{n^1_j})} \cdots a_k^{d_{j_1}R_k(\vect{n^1_j})}\right) \, \left(a_1^{d_{j_2}R_1(\vect{n^2_j})} \cdots a_k^{d_{j_2}R_k(\vect{n^2_j})}\right) \cdots \left(a_1^{d_{j_k}R_1(\vect{n^k_j})} \cdots a_k^{d_{j_k}R_k(\vect{n^k_j})}\right) \, W_{[\vect{N_{jk}}]}^{-1}\\
\vspace{0.1 cm}
= \left(A_1\,  W^{-1}_{d_{j_1}\vect{n^1_j}}\right)\, \left( W_{d_{j_1}\vect{n^1_j}} \, A_2\,  W^{-1}_{(d_{j_1}\vect{n^1_j} + d_{j_2}\vect{n^2_j})}\right) \cdots \left(W_{\vect{m}} \, A_k \, W^{-1}_{(d_{j_1}\vect{n^1_j} + \cdots + d_{j_k}\vect{n^k_j})} \right)\\
\vspace{0.1 cm}
= \left(A_1\,  W^{-1}_{d_{j_1}\vect{n^1_j}}\right)\, \left( W_{d_{j_1}\vect{n^1_j}} \, A_2\,  W^{-1}_{(d_{j_1}\vect{n^1_j} + d_{j_2}\vect{n^2_j})}\right) \cdots \left(W_{\vect{m}}\, A_k\, W^{-1}_{[\vect{N_{jk}}]} \right),
\end{array}
$$
where $\vect{m} = d_{j_1}\vect{n^1_j} + d_{j_2}\vect{n^2_j} + \cdots + d_{j_{k-1}}\vect{n^{k-1}_j}$ and $A_i = a_1^{d_{j_i}R_1(\vect{n^i_j})} \, a_2^{d_{j_i}R_2(\vect{n^i_j})} \cdots a_k^{d_{j_i}R_k(\vect{n^i_j})}$, for each $i = 1, \dots, k$. Also note that $d_{j_1}\vect{n^1_j} + \cdots + d_{j_k}\vect{n^k_j} = \vect{d_j}\vect{n_j} = [\vect{N_{jk}}]$.

First we explicitly write the term $A_1\,  W^{-1}_{d_{j_1}\vect{n^1_j}}$ of the above equation, precisely
\begin{equation}\label{eq: computation}
 (a_j^{[\vect{N_{jk}}]/\vect{n_j}} \, W_{[\vect{N_{jk}}]}^{-1})  = \left(A_1\,  W^{-1}_{d_{j_1}\vect{n^1_j}}\right)\, \left( W_{d_{j_1}\vect{n^1_j}} \, A_2\,  W^{-1}_{(d_{j_1}\vect{n^1_j} + d_{j_2}\vect{n^2_j})}\right) \cdots \left(W_{\vect{m}}\, A_k\, W^{-1}_{[\vect{N_{jk}}]} \right).
\end{equation}
And then we show that $A_1\,  W^{-1}_{d_{j_1}\vect{n^1_j}}$, in particular, each term in~\eqref{eq: computation} belongs to the announced subgroup $\bigl\langle X_{\vect{m,i}} \mid a_i \in V(\Gamma_\chi),\, 0 \leqslant \widetilde{\vect{m}} < kN^2R \bigr\rangle$.
$$
\begin{array}{l}
  A_1\,  W^{-1}_{d_{j_1}\vect{n^1_j}}    \\
  \vspace{0.1 cm}
  
 = \left(a_1^{d_{j_1}R_1(\vect{n^1_j})} \cdots a_k^{d_{j_1}R_k(\vect{n^1_j})}\right) \,  W^{-1}_{d_{j_1}\vect{n^1_j}}\\
  \vspace{0.1 cm}

  = \left( a_1^{d_{j_1}R_1(\vect{n^1_j})} \, W^{-1}_{d_{j_1}R_1(\vect{n^1_j})\vect{n_1}}\right) \, \left( W_{d_{j_1}R_1(\vect{n^1_j})\vect{n_1}} \, a_1^{d_{j_1}R_1(\vect{n^1_j})} \, W^{-1}_{d_{j_1}R_1(\vect{n^1_j})\vect{n_1} + d_{j_1}R_2(\vect{n^1_j})\vect{n_2}}\right) \cdots\\
  \vspace{0.1 cm}
  \cdots \left( W_{\vect{x}} \, a_k^{d_{j_1}R_k(\vect{n^1_j})} \, W^{-1}_{d_{j_1}(R_1(\vect{n^1_j})\vect{n_1} + \cdots + d_{j_1}R_k(\vect{n^1_j})\vect{n_k})}\right),  
\end{array}
$$
where $\vect{x} = d_{j_1} (R_1(\vect{n^1_j})\vect{n_1} + \cdots + R_{k-1}(\vect{n^1_j})\vect{n_{k-1}}$. The last equality holds as $d_{j_1}\vect{n^1_j} = d_{j_1}(R_1(\vect{n^1_j})\vect{n_1} + \cdots + d_{j_1}R_2(\vect{n^1_j})\vect{n_2})$.

Thus we have $|d_{j_1}|\,|R_1(\vect{n^1_j})|\, \widetilde{\vect{n_1}} \leqslant |d_{j_1}|\,(|R_1(\vect{n^1_j})|\, \widetilde{\vect{n_1}} +  |R_2(\vect{n^1_j})|\, \widetilde{\vect{n_2}}) \leqslant \cdots \leqslant |d_{j_1}|\,(|R_1(\vect{n^1_j})|\, \widetilde{\vect{n_1}} + \cdots +  |R_k(\vect{n^1_j})|\, \widetilde{\vect{n_k}})\leqslant kN^2R$ since $|d_{j_1}| \leqslant N_{jk} \leqslant N$, and $|R_i(\vect{n^1_j})|\widetilde{\vect{n_i}} \leqslant RN$, for each $i = 1, \ldots, k$.

Similarly, we can show that each term  $W_{\vect{y}} \, A_i \, W^{-1}_{\vect{y} + d_{j_i} \vect{n^i_j}} \in \bigl\langle X_{\vect{m,i}} \mid a_i \in V(\Gamma_\chi),\, 0 \leqslant \widetilde{\vect{m}} < kN^2R \bigr\rangle$, where $\vect{y} = d_{j_1} \vect{n^1_j} + \cdots + d_{j_{(i - 1)}} \vect{n^{(i - 1)}_j}$.

Altogether we have each term of~\eqref{eq: computation} is in $\bigl\langle X_{\vect{m,i}} \mid a_i \in V(\Gamma_\chi),\, 0 \leqslant \widetilde{\vect{m}} < kN^2R \bigr\rangle$.
Hence, $(a_j^{[\vect{N_{jk}}]/\vect{n_j}} \, W_{[\vect{N_{jk}}]}^{-1}) \in \bigl\langle X_{\vect{m,i}} \mid a_i \in V(\Gamma_\chi),\, 0 \leqslant \widetilde{\vect{m}} < kN^2R \bigr\rangle$. Similarly, we can show that $(a_k^{[\vect{N_{jk}}]/\vect{n_k}} \, W_{[\vect{N_{jk}}]}^{-1}) \in \bigl\langle X_{\vect{m,i}} \mid a_i \in V(\Gamma_\chi),\, 0 \leqslant \widetilde{\vect{m}} < kN^2R \bigr\rangle$ and the announced result follows.
\end{proof}

\begin{obs}\label{obs-1}
From Lemma \ref{lem:edges}, it follows that the vectorised weighted edges $\vect{e}_{j,k}^{\vect{s}}$, for $\vect{s}$ a multiple of $[\vect{N_{jk}}]$, also belong to the subgroup $\bigl\langle X_{\vect{m,i}} \mid \, a_i \in V(\Gamma_\chi), 0\leqslant \widetilde{\vect{m}} < kN^2R  \bigr\rangle$.  

Indeed, if $\vect{s} = \vect{d}[\vect{N_{jk}}]$, then $e_{j, k}^\vect{s} = 
(a_j^{\vect{d}[\vect{N_{jk}}]/\vect{n_j}} \, a_k^{-\vect{d}[\vect{N_{jk}}]/\vect{n_k}})$ and since $[a_j, a_k]=1$ in $G_\Gamma$, 
we have that $(a_j^{\vect{d}[\vect{N_{jk}}]/\vect{n_j}} \, a_k^{-\vect{d}[\vect{N_{jk}}]/\vect{n_k}}) =
(a_j^{[\vect{N_{jk}}]/\vect{n_j}} \, a_k^{-[\vect{N_{jk}}]/\vect{n_k}})^{\vect{d}}$ 
and $\psi(a_j^{[\vect{N_{jk}}]/\vect{n_j}} \, a_k^{-[\vect{N_{jk}}]/\vect{n_k}}) = \vect{0}$.
Therefore, $(a_j^{[\vect{N_{jk}}]/\vect{n_j}} \, a_k^{-[\vect{N_{jk}}]/\vect{n_k}})^{\vect{d}} = W_{d_1\vect{0}} \cdots W_{d_k\vect{0}} =  W_{\vect{0}} \cdots W_{\vect{0}} = 1$, and so it belongs to the mentioned subgroup 
trivially.
\end{obs}


Continuing the similar argument, we note down the following remark which we use to prove our main results.

\begin{rem}\label{rem: general path}
Let $p^{\vect{s}}_{i,j} = e^{\vect{s}}_{i_1, i_2} \, e^{\vect{s}}_{i_2, i_3} \dots e^{\vect{s}}_{i_d, i_{d+1}}$ be a vectorised weighted path in $\Gamma_\chi$, where $\vect{s} = (s_1, \ldots, s_k) \in \ZZ^k$ and $s_i$ is a multiple of $N$ for each $i = 1, \ldots, k$, i.e., $\vect{s} = \vect{m}[\vect{N}]$ for some $\vect{m} \in \ZZ^k$. First, we note that the path $p^{\vect{s}}_{ij}$ is well-defined as $\vect{n_{i_k}}$ divides $\vect{s}$ for $k = 1, \dots, d+1$ --- which follows from the fact that each $s_i$ is a multiple $N$. since $p^{\vect{s}}_{i,j} = e^{\vect{s}}_{i_1, i_2} \, e^{\vect{s}}_{i_2, i_3} \dots e^{\vect{s}}_{i_d, i_{d+1}}$, and each $ e^{\vect{s}}_{i_j, i_k} \in \bigl\langle X_{\vect{m,i}} \mid 0 \leqslant \widetilde{\vect{m}} \leqslant kN^2R\bigr\rangle$ from~\ref{obs-1}, the remark follows directly.

Also, for any vertices $a_j, a_k$ which is joined by the path $p_{i,j}$, it is straightforward that $\psi(p^N_{i,j}) = \vect{0}$. Hence, $p^{\vect{s}}_{i,j}    = p^{\vect{d}[\vect{N}]}_{i,j} = W_{d_1\vect{0}} \cdots W_{d_k\vect{0}} =  W_{\vect{0}} \cdots W_{\vect{0}} = 1$, and belongs to the subgroup
$\bigl\langle X_{\vect{m,i}} \mid 0 \leqslant \widetilde{\vect{m}} \leqslant kN^2R\bigr\rangle$ trivially.
\end{rem}

More generally, we obtain the following.

\begin{lem}\label{lem: better bound}
   Let $p = p_{j,k}$ be any path in $\Gamma$ and let $\vect{s} = (s_1, \ldots, s_k)$ be an element of $\ZZ^k$ where each $s_i$ is divisible by $N(p)$. Then the vectorised weighted path $p_{j,k}^{\vect{s}} \in \bigl\langle X_{\vect{m,i}} \mid a_i \in V(\Gamma_\chi), \, 0\leqslant \widetilde{\vect{m}} < kN^2R \bigr\rangle$.
\end{lem}

\begin{proof}
   This follows from the definition of a weighted path, Lemma~\ref{lem:edges}, Observation~\ref{obs-1} and Remark~\ref{rem: general path}.
\end{proof}

Recall that $R = \underset{j}{\max}\,\{ R_j(\vect{v}) \mid \vect{v} \in \mathcal{N}\}$, and $R>0$. Also recall that $\mathcal{M} = \{ \vect{v} \in \ZZ^k \mid \widetilde{\vect{v}} \leqslant \widetilde{\vect{m}} \}$ and $R_{\mathcal{M}} = \underset{j}{\max}\,\{ R_j(\vect{v}) \mid \vect{v} \in \mathcal{M}\}$.


\begin{rem}\label{rem:relations_in_generators}
Notice that in the family of relations $R_3$ in Theorem~\ref{thm:finite_presentation} (see also Theorem(Finite presentation) in Section~\ref{intro}), the elements $X_{\vect{t,i}}$ for $k N^2 R   \leqslant \widetilde{\vect{t}} < k^2N^3R_{\mathcal{M}}R $ are not in the generating set and so formally, $X_\vect{{t,i}}$ should be replaced by a word in the generators that represents it. We abuse the notation and keep $X_\vect{{t,i}}$ in the relations for simplicity.  
\end{rem}

\begin{thm}[Finite generation]\label{thm:finite_generation}
    Let $G_\Gamma$ be a RAAG and $\psi \colon G_\Gamma \twoheadrightarrow \ZZ^k$ be an epimorphism and $\ker \psi=\gh$. Then,  $\gh$ is finitely generated if and only if $\Gamma$ is $k$-connected. More precisely,
\[
\gh =\bigl\langle X_{\vect{m,i}} \, \, \mid \, \, X_{\vect{t,i}}=\tau(W_\vect{t} \, a_i \, W_{\vect{t}+\vect{n_i}}^{-1}), \tau(W_\vect{M}\, [a_i,a_j] \,W_{\vect{M}}^{-1})=1 \bigr\rangle,
\]
where $X_{\vect{d,i}} = W_\vect{d} \, a_i \, W_{\vect{d} + \vect{n_i}}^{-1}$, $0 \leqslant \widetilde{\vect{m}} < kN^2R$, $0\leqslant \widetilde{\vect{t}} < k^2N^3R_{\mathcal{M}}R $, $\vect{M}\in \ZZ^k$, $a_i\in V(\Gamma)$, and $(a_i, a_j) \in E(\Gamma)$.
\end{thm}
\begin{proof}

We first show that $\gh$ is generated by $\bigl\langle X_{\vect{m,i}} \mid 0\leqslant \widetilde{\vect{m}} <kN^2R, \, a_i \in V(\Gamma)\bigr\rangle$ and when expressing $\tau(W_{\vect{M}} \, a_i\, W_{\vect{M} + \vect{n_i}}^{-1})$ for $\widetilde{\vect{M}} \geqslant k^2N^3R_{\mathcal{M}}R $ in terms of the finite set of generators, we will see that the generator $X_{\vect{M,i}}$ does not appear in this expression and so it can be removed together with the relation $X_{\vect{M,i}} = \tau(W_{\vect{M}} \, a_i \, W_{\vect{M} + \vect{n_i}}^{-1})$, using a Tietze move.

Recall that
\begin{align*}
 X_{\vect{m,i}} &\enspace = W_{\vect{m}} \, a_i \, (W_{\vect{m} + \vect{n_i}})^{-1}  \\
%
%
&\enspace = a_{1}^{R_1(\vect{m})} \, \, a_{2}^{R_2(\vect{m})} \cdots a_{k}^{R_k(\vect{m})} \, \, a_i \, \, (W_{\vect{m} + \vect{n_i}})^{-1}.
\end{align*}
%

Divide $R_j(\vect{m})$ by $N$ with remainder, $R_j(\vect{m}) = q_jN + s_j$; $j= 1, \ldots, k$, here $s_j < N$. Rewriting the last equation we have
\begin{equation}\label{equ: start equ}
    \begin{array}{l}
    \vspace{0.1cm}
         W_{\vect{m}} \, a_i \, (W_{\vect{m} + \vect{n_i}})^{-1} \\
         
         \vspace{0.1cm}
         
         = a_{1}^{R_1(\vect{m})} \, \, a_{2}^{R_2(\vect{m})} \cdots a_k^{R_k(\vect{m})} \, \, a_i \, \, (W_{\vect{m} + \vect{n_i}})^{-1}\\
         
         \vspace{0.1cm}
         
         = a_{1}^{s_1} \, \, a_{1}^{q_1N} \, \, a_{2}^{s_2} \, \, a_{2}^{q_2N} \cdots  a_k^{s_k} \, \, a_k^{q_k N } \, \,  a_i \, \, (W_{\vect{m} + \vect{n_i}})^{-1}\\
         \vspace{0.1cm}
         
         = a_{1}^{s_1} \left(a_{1}^{q_1N} \, a_{2}^{-(Nq_1\vect{n_1})/\vect{n_2}}\right) \, \, a_{2}^{s_2} \left(a_{2}^{N(q_1\vect{n_1}+ q_2 \vect{n_2})/\vect{n_2}} \, a_3^{-N(q_1\vect{n_1}+ q_2\vect{n_2})/\vect{n_3}}\right) \cdots \\
         \quad \quad \quad \quad \cdots a_k^{s_k} \, \, a_k^{N\vect{d}/\vect{n_k}} \, \, a_i \, \, (W_{\vect{m} + \vect{n_i}})^{-1},
         

\end{array}
\end{equation}
 where $\vect{d}= q_1 \vect{n_1} + \cdots + q_k \vect{n_k}$. Note that $\vect{m}N = \vect{m} [\vect{N}]$, for any $\vect{m} \in \ZZ^k$ and $N \in \ZZ$, and $[\vect{N}] = (N, N, \ldots, N)$. 
Then applying the notation of \emph{vectorised weighted path}, we have 
\begin{align*}
  \left(a_{1}^{q_1N } \, a_{2}^{-(Nq_1\vect{n_1})/\vect{n_2}}\right) &\enspace = \left( a_1^{(q_1\vect{n_1}[\vect{N}])/\vect{n_1}} \, a_2^{-(q_1\vect{n_1}[\vect{N}])/\vect{n_2}} \right) \\
&\enspace = \left( a_1^{[\vect{N}]/\vect{n_1}} \, a_2^{-[\vect{N}]/\vect{n_2}} \right)^{q_1\vect{n_1}} \\
&\enspace = p_{1,2}^{Nq_1\vect{n_1}}.
\end{align*}
%
The $p_{\ell,j}$ are paths in $\Gamma_\chi$. Notice that the paths $p_{\ell,\ell+1}$, $\ell=1, \dots, k-1$. Altogether, Equation~\eqref{equ: start equ} becomes,
$$
\begin{array}{l}
  W_{\vect{m}} \, a_i \, W_{\vect{m} + \vect{n_i}}^{-1}\\
  \vspace{0.15cm}
  
  = a_{1}^{s_1} \, \, p_{1,2}^{Nq_1\vect{n_1}} \, \, a_{2}^{s_2} \, \, p_{2,3}^{N(q_1\vect{n_1} + q_2\vect{n_2})} \, \, a_{3}^{s_3} \cdots a_k^{s_k} \, \, \left(a_k^{\vect{d}[\vect{N}]/\vect{n_k}} \, \, a_j^{-\vect{d}[\vect{N}]/\vect{n_j}}\right) \, \, a_i \, \, a_j^{\vect{d}[\vect{N}]/\vect{n_j}} \, \,W_{\vect{m} + \vect{n_i}}^{-1}\\
  \vspace{0.15cm}

  = a_{1}^{s_1} \, \, p_{1,2}^{Nq_1\vect{n_1}} \, \, a_{2}^{s_2} \, \, p_{2,3}^{N(q_1\vect{n_1} + q_2\vect{n_2})} \, \, a_{3}^{s_3} \cdots a_k^{s_k} \, \, p_{k, j}^{N\vect{d}} \, \, a_i \, \, a_j^{\vect{d}[\vect{N}]/\vect{n_j}} \, \,W_{\vect{m} + \vect{n_i}}^{-1}.      
\end{array}
$$
Note that there exist $a_j\in V(\Gamma_\chi)$ with either $a_i=a_j$ or $(a_i, a_j) \in \Gamma$ and the path $p_{k,j}$ in $\Gamma_\chi$. The existence of $a_j \in V(\Gamma_\chi)$ and the path $p_{k,j}$ in $\Gamma_\chi$ is derived from Construction~\ref{cons 1} which proves that $\Gamma$ is $k$-connected implies $\Gamma_\chi$ is connected and $0$-acyclic-dominating.

Recall from Remark~\ref{rem: on R} that for every $\vect{m} \in \ZZ^k$, $R_j(\vect{m} + \vect{n_i}) = R_j(\vect{m}) + R_j(\vect{n_i})$ for each $i =  1, \ldots, k$ and note that
$$
\begin{array}{l}
   a_i \, \, a_j^{\vect{d}[\vect{N}]/\vect{n_j}} \, \, (W_{\vect{m} + \vect{n_i}})^{-1}\\
   \vspace{0.15 cm}
   
   =  a_i \, \, a_j^{\vect{d}[\vect{N}]/\vect{n_j}} \, \, a_k^{-R_k(\vect{m} + \vect{n_i})} \cdots a_2^{-R_2(\vect{m} + \vect{n_i})} \, \, a_1^{-R_1(\vect{m} + \vect{n_i})}\\
   \vspace{0.1 cm}

   = a_i\, \,  a_j^{\vect{d}[\vect{N}]/\vect{n_j}} \, \, a_k^{-R_k(\vect{m})} \, \, a_k^{-R_k(\vect{n_i})} \cdots a_1^{-R_1(\vect{m})} \, \, a_1^{-R_1(\vect{n_i})}\\
   \vspace{0.1 cm}

   = a_i\, \,  a_j^{\vect{d}[\vect{N}]/\vect{n_j}} \, \, a_k^{-q_kN}\, \, a_k^{-(s_k+R_k(\vect{n_i}))} \, \, \cdots a_1^{-q_1N} \, \, a_1^{-(s_1 + R_1(\vect{n_i}))}\\
   \vspace{0.1 cm}

   = a_i \, \, a_j^{\vect{d}[\vect{N}]/\vect{n_j}} \, \, a_k^{-\vect{d}[\vect{N}]/\vect{n_k}} \, \,  a_k^{-(s_k + R_k(\vect{n_i}))} \, \, \left( p_{k-1, k}^{N\vect{c}}\right)^{-1} \cdots a_1^{-q_1N} \, \, a_1^{-(s_1 + R_1(\vect{n_i}))}\\
   \vspace{0.1 cm}

   = a_i \, \, \left(p_{k, j}^{N\vect{d}}\right)^{-1} a_k^{-(s_k + R_k(\vect{n_i}))} \cdots \left( p_{1,2}^{Nq_1\vect{n_1}} \right)^{-1}\, \, a_1^{-(s_1 + R_1(\vect{n_i}))},
\end{array}
$$
where $\vect{c} = \vect{d} - q_k\vect{n_k} = q_1\vect{n_1} + \cdots + q_{k-1}\vect{n_{k-1}}$ and note that $a_k^{-q_k N + \frac{\vect{d}[\vect{N}]}{\vect{n_k}}} = a_k^{\frac{\vect{c[\vect{N}]}}{\vect{n_k}}}$.
Finally, considering $\vect{d}' = \vect{n_1}s_1 + \cdots + \vect{n_k} s_{k}$ we arrive at the following:

\begin{equation} \label{eq:maineq}
\begin{array}{l}  
 W_{\vect{m}} \, a_i \, W_{\vect{m} + \vect{n_i}}^{-1} \\
 \vspace{0.1 cm}

 =  a_{1}^{s_1} \, \, p_{1,2}^{Nq_1\vect{n_1}} \, \, a_{2}^{s_2} \, \, p_{2,3}^{N(q_1\vect{n_1} + q_2\vect{n_2})} \, \, a_{3}^{s_3} \cdots a_k^{s_k} \, \, p_{k, j}^{N\vect{d}} \, \, a_i \, \, \left(p_{k, j}^{N\vect{d}}\right)^{-1} a_k^{-(s_k + R_k(\vect{n_i}))} \cdots \\
          \cdots \left( p_{1,2}^{Nq_1\vect{n_1}} \right)^{-1}\, \, a_1^{-(s_1 + R_1(\vect{n_i}))}\\
  \vspace{0.1 cm}

= \left( a_{1}^{s_1} \, \, W_{\vect{n_1}s_1} \right)^{-1}  \left(W_{\vect{n_1}s_1} \, \, p_{1, 2}^{Nq_1\vect{n_1}} \, \, W_{\vect{n_1}s_1}^{-1}\right)
        \left(W_{\vect{n_1}s_1} \, \, a_{2}^{s_2} \, \, W_{\vect{n_1}s_1+\vect{n_2}s_2}^{-1}\right)
        \cdots \\
        \cdots \left(W_{\vect{d}'- \vect{n_k} s_k} \, \, a_k^{s_k} \, \, W_{\vect{d}'}^{-1}\right)\, \,
        \left( W_{\vect{d}'} \, \, p_{k, j}^{N \vect{d}} \, \, W_{\vect{d}'}^{-1}\right) \, \, \left(W_{\vect{d}'} \, \, a_i \, \, W_{\vect{n_i} + \vect{d}'}^{-1}\right) \\  
        \left(W_{\vect{n_i} + \vect{d}'} \, \, (p_{k, j}^{N\vect{d}})^{-1} W_{\vect{n_i} + \vect{d}'}^{-1}\right) \, \, 
        \left(W_{\vect{n_i} + \vect{d}'}\, \, a_k^{-(s_k + R_k(\vect{n_i}))} \, \, W_{\vect{n_i} + \vect{d}'- \vect{n_k} s_k - \vect{n_k}   R_k(\vect{n_i})}^{-1}\right) \cdots\\
        \cdots \left(W_{R_1(\vect{n_i})\vect{n_1} + \vect{n_1}s_1} \, \, a_{1}^{-(s_1+R_1(\vect{n_i}))} \, \, W_{(R_1(\vect{n_i})\vect{n_1}+\vect{n_1}s_1) -\vect{n_1}s_1-R_1(\vect{n_i})\vect{n_1}}^{-1}\right).
\end{array}
\end{equation}
Every term in the last expression of Equation \eqref{eq:maineq} has one of the two following forms:
    \begin{enumerate}
        \item[(A)]\label{it:1} $W_{\vect{x}}\, a_i^{y} \, W_{\vect{x}+\vect{n_i}y}^{-1}$, or
       \item[(B)]\label{it:2} $W_{\vect{x}}\, p_{\ell,t}^{\vect{z}}\,  W_{\vect{x}}^{-1}$, 
        $W_{\vect{x}+\vect{n_i}y}\, p_{\ell, t}^{\vect{z}}\, W_{\vect{x}+\vect{n_i}y}^{-1}$, where $p_{\ell,t}^{\vect{z}}$ is a vectorised weighted path in $\Gamma_\chi$ joining any two vertices $a_\ell, a_t \in \{ a_1, \dots, a_k, a_j\}$ and $\vect{z} \in \ZZ^k$ is a multiple of $[\vect{N}]$.
    \end{enumerate}
Furthermore, using the definition of $\vect{d}'$, and since for any $\vect{m} \in \ZZ^k$ and any integer $k$  $\widetilde{\vect{m}k} = \widetilde{\vect{m}}|k|$, we get the following
$$
\widetilde{\vect{x}}, \widetilde{\vect{x}} + \widetilde{\vect{n_i}}y \leqslant \max_{1\leqslant \ell \leqslant k} \bigl\{ \widetilde{\vect{n_1}}|s_1| + \widetilde{\vect{n_2}} |s_2| \ldots + \widetilde{\vect{n_\ell}}|s_\ell| + \sum_{\ell \leqslant j \leqslant k}| R_j(\vect{n_i})|\widetilde{\vect{n_j}}\bigl\}. 
$$
Since each $\widetilde{\vect{n_i}}|s_i| < N^2 \leqslant N^2R$, and $|R_j(\vect{n_i})| \widetilde{\vect{n_j}} \leqslant R_i N \leqslant RN \leqslant N^2R$ for each $1\leqslant i,j\leqslant k$, we can conclude that
$$
\widetilde{\vect{x}}, \widetilde{\vect{x}} + \widetilde{\vect{n_i}}y < k N^2 R .
$$
Write each term $W_{\vect{x}} \, a_i^{y} \, (W_{\vect{x}+\vect{n_i}y})^{-1}$ as follows
\begin{equation}\label{eq:first_type}
\begin{array}{l}
    W_{\vect{x}} \, a_i^{y} \, (W_{\vect{x}+\vect{n_i}y})^{-1}   \\
     = \left(W_{\vect{x}}\, a_i \, W_{\vect{x} + \vect{n_i}}^{-1}\right) \left(W_{\vect{x} + \vect{n_i}} \, a_i \, (W_{\vect{x} + \vect{2n_i}})^{-1}\right) \cdots \left(W_{\vect{x}+(y-1)\vect{n_i}} \, a_i \, (W_{\vect{x}+\vect{n_i}y})^{-1}\right). 
\end{array}
\end{equation}
\paragraph{\underline{Type (A) terms}} Terms of the form $W_{\vect{x}} \, \vect{a_i}^{y} \, (W_{\vect{x}+\vect{n_i}y})^{-1}$ in Equation \eqref{eq:maineq} can be written as a product of generators from the theorem as in Equation \eqref{eq:first_type}, since in that case, $\widetilde{\vect{x}}$ and $\widetilde{\vect{x}} + \widetilde{\vect{n_i}} y$ are bounded by $kN^2R$. 

\paragraph{\underline{Type (B) terms}} In this case, we write $p_{\ell, t}^{\vect{z}}$ as a product of weighted edges, $e_{\ell_1, \ell_2}^{\vect{z}} = (\vect{a_{\ell_1}}^{\vect{z}/\vect{n_{\ell_1}}} \vect{a_{\ell_2}}^{-\vect{z}/\vect{n_{\ell_2}}})$ and we write $W_{\vect{x}} \, p_{\ell, t}^{\vect{z}} \, W_{\vect{x}}^{-1}$ as a product of conjugates of vectorised weighted edges $W_{\vect{x}} \, e_{\ell_1, \ell_2}^{\vect{z}} W_{\vect{x}}^{-1}$.

Note that in Equation \eqref{eq:maineq} $\vect{z}$ is always a multiple of $[\vect{N}]$, say $\vect{z} = \vect{m}[\vect{N}]$. 


Now we show that $W_{\vect{x}}\, e_{\ell_1, \ell_2}^{\vect{m}N} W_{\vect{x}}^{-1}$ can be written as a product of the generators from the statement of the theorem and then the result follows from . 

Write
$$
W_{\vect{x}}\, e_{\ell_1, \ell_2}^{\vect{m}N} W_{\vect{x}}^{-1} = \left(W_{\vect{x}}\,  a_{\ell_1}^{\vect{m}[\vect{N}]/\vect{n_{\ell_1}}}\, W_{\vect{x} + \vect{m}[\vect{N}]}^{-1}\right) \left(W_{\vect{x} + \vect{m}[\vect{N}]}\, a_{\ell_2}^{-\vect{m}[\vect{N}]/\vect{n_{\ell_2}}} \, W_{\vect{x}}^{-1}\right).
$$

Since $\widetilde{\vect{x}}, \, \widetilde{\vect{x}} + \widetilde{\vect{n_i}}y < kN^2R$, type (B) terms $W_{\vect{x}}\, p_{\ell, t}^{\vect{z}} \,W_{\vect{x}}^{-1}$ and $W_{\vect{x}+\vect{n_i}y} \, p_{j,k}^{\vect{z}} \, \left(W_{\vect{x}+\vect{n_i}y}\right)^{-1}$ in Equation \eqref{eq:maineq} can be written as a product of the generators $X_{\vect{m,i}}$ for $0\leqslant \widetilde{\vect{m}} < kN^2R$ and $a_i\in V(\Gamma)$ (in the free group $F(a_i\mid a_i\in V(\Gamma))$).

From Lemma \ref{lem:trasverse_homomorphism}, we have that $\tau(W_{\vect{M}} \, a_i W_{\vect{M} + \vect{n_i}}^{-1})$ is the same as the product of some $\tau(W_{\vect{m}} {a_i} W_{\vect{m} + \vect{n_i}}^{-1})$ where $\widetilde{\vect{m}} < kN^2R$. 

Furthermore, 
\begin{equation} \label{eq:maineq 2}
\begin{array}{l}  
\vspace{0.1cm}
      \tau\left(W_{\vect{m}} \, a_i \,(W_{\vect{m} + \vect{n_i}})^{-1}\right)\\
      \vspace{0.1cm}


        = \tau\left(a_{1}^{R_1(\vect{m})}  \cdots a_k^{R_k(\vect{m})}  \, a_i \,  a_k^{-R_k(\vect{m} + \vect{n_i})} \cdots a_{1}^{-R_1(\vect{m}+ \vect{n_i})}\right)\\
        \vspace{0.1cm}

        =\left(a_1 \, W_{\vect{n_1}}^{-1}\right) \cdots \left(W_{R_1(\vect{m})\vect{n_1} -\vect{n_1}}\, a_1 \,W_{R_1(\vect{m})\vect{n_1}}^{-1}\right)\\
        \vspace{0.1cm}
        
        \hspace{0.35cm}\left(W_{R_1(\vect{m})\vect{n_1}} a_2 W_{R_1(\vect{m})\vect{n_1} + \vect{n_2}}^{-1}\right) \cdots 
        \left( W_{R_1(\vect{m})\vect{n_1} + R_2(\vect{m})\vect{n_2}-\vect{n_2}}\,  
        a_2\,W_{R_1(\vect{m})\vect{n_1} + R_2(\vect{m})\vect{n_2}}^{-1} \right) \cdots \\
        \vspace{0.1cm}

        \hspace{0.35cm} \cdots \left(W_{R_1(\vect{m})\vect{n_1} + \cdots +R_{k -1}(\vect{m})\vect{n_{k-1}}} \, a_k \, W_{R_1(\vect{m})\vect{n_1} + \cdots +R_{k -1}(\vect{m})\vect{n_{k-1}} + \vect{n_{k}}}^{-1}\right) \cdots \\
        \vspace{0.1cm}

        \hspace{0.35cm}\cdots \left(W_{R_1(\vect{m})\vect{n_1} + \cdots +R_{k}(\vect{m})\vect{n_{k}} - \vect{n_k}} \, a_k \, W_{R_1(\vect{m})\vect{n_1} + \cdots +R_{k}(\vect{m})\vect{n_{k}}}^{-1}\right)
        \\
        \vspace{0.1cm}

        \hspace{0.35cm} \left(W_{\vect{m}}\, a_i \, W_{\vect{m} +\vect{n_i}}^{-1}\right) \, \left(W_{\vect{m} +\vect{n_i}} a_{k}^{-1} \, W_{\vect{m} + \vect{n_i} - \vect{n_k}}^{-1} \right) \cdots (W_{\vect{n_1}} a_1^{-1} W^{-1}_{\vect{0}}) \\
        \vspace{0.1cm}

        = X_{\vect{0,1}} \cdots X_{\vect{R_1(m)n_1-n_1,1}}\cdots 
        X_{\vect{m-n_k, k}} \, X_{\vect{m,i}} \cdots X_{\vect{0,1}}^{-1},
        \end{array}
\end{equation}

observe that the second last equality holds as $\vect{m} = R_1(\vect{m})\vect{n_1} + \cdots +R_{k}(\vect{m})\vect{n_{k}}$. Recall that $R_\mathcal{M} = \underset{j}{\max}\,\{ R_j(\vect{v}) \mid \vect{v} \in \mathcal{M}\}$.
Notice that for all the generators $X_{\vect{t,i}}$ in the above equation, we have that $\widetilde{\vect{t}}\leqslant k R_\mathcal{M} N \widetilde{\vect{m}}$ and since $\widetilde{\vect{m}} \leqslant kN^2R$, we conclude that $\widetilde{\vect{t}} \leqslant k^2N^3 R_\mathcal{M}R$.

It follows that the generators $X_{\vect{M,i}}$ do not appear in the expression of $\tau\left(W_{\vect{m}} \, a_i \,(W_{\vect{m} + \vect{n_i}})^{-1}\right)$ for $\widetilde{\vect{M}} \geqslant k^2N^3 R_\mathcal{M} R $, therefore, we can apply a Tietze transformation and remove the generator and the relation and the result follows.

 First we observe that, for any induced subgraph $\gamma \leqslant \Gamma$, we have that $G_\Gamma$ retracts to $G_\gamma$ and the epimorphism $\psi:G_\Gamma \to \ZZ^k$ factors through the epimorphism $\psi|_{G_{\gamma}}: G_{\gamma} \to \ZZ^k$ given by the restriction. 
 Let us assume that $\Gamma$ is not $k$-connected. Without the loss of generality let $\{ a_1, \ldots, a_{k-1}\}$ be the set of vertices such that $\Gamma \setminus \{ a_1, \ldots, a_{k-1}\}$ is not connected. Also let $\gamma$ be the induced subgraph of $\Gamma$ generated by $V(\Gamma)\setminus \{ a_1, \ldots, a_{k-1} \}$.
 As $\gamma$ is not connected, $G_{\gamma}$ is a nontrivial free product. Since by a result by Baumslag, see \cite{Baumslag}, nontrivial finitely generated normal subgroups of free products are of finite index, we have that the kernel $K_\gamma$ of the epimorphism $\psi: G_\gamma \to \ZZ^k$ is not finitely generated and since the epimorphism $G_\Gamma \to G_\gamma$ induces an epimorphism of kernels $\gh \twoheadrightarrow K_\gamma$, we conclude that $\gh$ is not finitely generated. 
 \end{proof}

\begin{thm}[Finite presentation]\label{thm:finite_presentation}

 Let $G_\Gamma$ be a right-angled Artin group and $\psi \colon G_\Gamma \twoheadrightarrow \ZZ^k$ be an epimorphim. Then, the kernel $\ker \psi=\gh$ is finitely presented if and only if $\triangle_\Gamma$ is $(k-1)$-$1$-connected. Moreover, in this case the explicit presentation of $\gh$ is as follows:
%
%
    \[
    \bigl\langle X_{\vect{m,i}} \, \mid \, R_1, R_2, R_3\bigr\rangle, 
    \]
    where $X_\vect{{m,i}} = W_{\vect{m}} \, a_i \, (W_{\vect{m} + \vect{n_i}})^{-1}$, $0 \leqslant \widetilde{\vect{m}} < kN^2R$, and the sets of relations $R_1, R_2, R_3$ are defined as follows:  
\begin{itemize}
\item[$R_1$:] for any directed $3$-cycle $(a_{i_1}, a_{i_2}, a_{i_3})$,
\[
 \tau(W_{\vect{s}}\,
  e_{i_1,i_2}^{Nd} e_{i_2,i_3}^{Nd}  e_{i_3,i_1}^{Nd}  \, W_{\vect{s}}^{-1}), \hspace{0.2cm} 0 \leqslant \widetilde{\vect{s}} < N, d=\pm 1, \text{ where } e_{i_r,i_p}=a_{i_r}a_{i_p}^{-1};
\]
\item[$R_2$:] $\tau(W_{\vect{s}}\, [a_i, a_j] \, W_{\vect{s}}^{-1})=1, \hspace{0.2cm} 0 \leqslant \widetilde{\vect{s}} < N$, where $(a_i, a_j) \in E(\Gamma);$
\item [$R_3$:] $X_{\vect{t,i}}=\tau(W_{\vect{t}}\, a_i (W_{\vect{t} + \vect{n_i}})^{-1})$ for $0\leqslant \widetilde{\vect{t}} < k^2N^3R_{\mathcal{M}}R$.
\end{itemize}
\end{thm}

\begin{proof}
First we assume that $\triangle_\Gamma$ is $(k-1)$-$1$-connected and so $\Gamma$ is $k$-connected, from Theorem \ref{thm:finite_generation} we have that the kernel is finitely generated by $W_{\vect{m}}\, a_i\, (W_{\vect{m} + \vect{n_i}})^{- 1}$, where $0 \leqslant \widetilde{\vect{m}} < kN^2R$.

Let us consider the following relation:

\begin{itemize}
    \item[$R'_1$:] for any directed cycle $(a_{i_1}, a_{i_2}, \ldots, a_{i_\ell}=a_{i_1})$ with $\ell \geqslant 3$,
\[
 \tau(W_{\vect{s}}\,
 e_{i_1,i_2}^{N\vect{d}}\, e_{i_2,i_3}^{N\vect{d}} \cdots e_{i_{\ell-1},i_\ell}^{N\vect{d}}  \, W_{\vect{s}}^{-1})=1, \hspace{0.2cm} 0 \leqslant \widetilde{\vect{s}} < N, \vect{d}\in \ZZ^k;
\]
\end{itemize}

From Observation \ref{obs-1} we note that all relations from $R'_1$ can be written as words in the generators of $\gh$.
We prove that $\gh$ admits a presentation with relations $R_1, R_2, R_3$ by showing that $R_1$ is a consequence of $R'_1$ (together with $R_2$ and $R_3$) and the normal subgroup generated by the Reidemeister--Schreier relations $\{W_{\vect{M}} \, [a_i,a_j]\,  W_{\vect{M}}^{-1}\}$ given in Theorem \ref{thm: R-S presentation} are contained in the normal closure of the relations $R'_1$ and $R_2$ in the free group $F(a_i\mid a_i\in V(\Gamma))$.


   Let $\vect{M} = (M_1, \ldots, M_k)$ and recall that $[\vect{N}] = (N, \ldots, N)$. Applying the division algorithm for each coordinate, we get $M_i = d_i N + s_i$, $s_i < N$ for each $i = 1, \ldots, k$.  Let us denote $\vect{d} [\vect{N}] = \vect{N'}$. Clearly, $\vect{N'} = \vect{d} N$. Therefore, $\vect{M} = \vect{d} N + \vect{s} = \vect{N'} + \vect{s}$.

   Also note that $W_{\vect{d}} = a_1^{R_1(\vect{d})} \cdots a_k^{R_k(\vect{d})}$. Recall fom Remark~\ref{rem: on R} that $R_j (\vect{m} + \vect{n}) = R_j (\vect{m}) + R_j(\vect{n})$ for any $\vect{m}, \vect{n} \in \ZZ^k$ and $j = 1, \ldots, k$ (see~\ref{rem: on R}). On the one hand, we have $W_{\vect{N'}} = W_{N\vect{d}} = a_1^{R_1(N\vect{d})} \cdots a_k^{R_k(N\vect{d})} = a_1^{NR_1(\vect{d})} \cdots a_k^{NR_k(\vect{d})}$. On the other hand, $W_{\vect{N'}} = a_1^{R_1(\vect{N'})} \cdots a_k^{R_k(\vect{N'})}$. Comparing these two expressions we get $R_j(\vect{N}') = N R_j(\vect{d})$ for all $j = 1, \ldots, k$.


     Let us consider the case when $\vect{M} = \vect{N'} = \vect{d} N$. Using the expression $ W_{\vect{d} N} = a_1^{NR_1(\vect{d})} \cdots a_k^{NR_k(\vect{d})}$ we get the following:
$$
\begin{array}{l}
    \vspace{0.1cm}
      W_{\vect{d} N} \, [a_i, a_j] \, W_{\vect{d} N}^{-1}  \\
       \vspace{0.1cm}
       
        = a_1^{NR_1(\vect{d})} \cdots a_k^{NR_k(\vect{d})} \, \, [a_i, a_j] \, \, a_k^{-NR_k(\vect{d})} \cdots a_1^{-NR_1(\vect{d})}\\
        \vspace{0.1cm}

        = \left(a_{1}^{(NR_1(\vect{d})\vect{n_1})/\vect{n_1}} \, \, a_{2}^{-(NR_1(\vect{d})\vect{n_1})/\vect{n_2}}\right) \cdots
        \left(a_{k -1}^{N(R_1(\vect{d})\vect{n_1} + \cdots R_{k -1}(\vect{d}) \vect{n_{k -1}})/\vect{n_{k-1}}} \, \, a_k ^{-N(R_1(\vect{d})\vect{n_1} + \cdots R_{k -1}(\vect{d}) \vect{n_{k -1}})/\vect{n_k}}\right)\\
        
        \left(a_k ^{N(R_1(\vect{d})\vect{n_1} + \cdots R_k(\vect{d}) \vect{n_k})/\vect{n_k}} \, \, a_\ell^{-N(R_1(\vect{d})\vect{n_1} + \cdots R_k(\vect{d}) \vect{n_k})/\vect{n_\ell}}\right) \, \, [a_i, a_j] \, \, a_\ell^{N\vect{d}/\vect{n_\ell}} \, \, W_{\vect{d} N}^{-1}.
        
    \end{array}
$$

Recall that $R_1(\vect{d})\vect{n_1} + \cdots + R_k(\vect{d}) \vect{n_k} = \vect{d}$. Using the definition of vectorised weighted path, we have that:
$$
 W_{\vect{d} N} \, [a_i, a_j] \, W_{\vect{d} N}^{-1}            
        = p_{1,2}^{NR_1(\vect{d})\vect{n_1}} \, \, p_{2,3}^{N(R_1(\vect{d})\vect{n_1} +R_2(\vect{d}) \vect{n_2})} \cdots p_{{k -1},k}^{N(\vect{d}-R_k(\vect{d})\vect{n_k})} \, \, p_{k, \ell}^{N\vect{d}} \, \, [a_i, a_j] \, \, p_{\ell,1}^{N\vect{d}} \, \, p_{1, k}^{N\vect{d}} \, \, X^{-1},
$$
where $X= p_{1,2}^{NR_1(\vect{d})\vect{n_1}} \, \, p_{2,3}^{N(R_1(\vect{d})\vect{n_1} +R_2(\vect{d}) \vect{n_2})} \cdots p_{{k -1},k}^{N(\vect{d}-R_k(\vect{d})\vect{n_k})}$. Now we consider $\Gamma_\chi$ as of Construction~\ref{cons 1}, then
the paths $p_{r,s}$ are in $\Gamma_\chi$, $a_\ell\in V(\Gamma_\chi)$ and $(a_\ell, a_i), (a_\ell, a_j)\in E(\Gamma)$. 
The paths $p_{r,s}$ and the vertex $a_\ell$ with the required properties exist since $\triangle_{\Gamma}$ is $(k-1)$-$1$-connected and so $\triangle_{\Gamma_\chi}$ is $1$-connected and $1$-acyclic-dominating (see Fact~\ref{fact 1} and Construction~\ref{cons 1}).
Hence,
\[
 W_{\vect{d} N} \, [a_i, a_j] \, W_{\vect{d} N}^{-1} = X \, \, p_{k, \ell}^{N\vect{d}} \, \, [a_i, a_j] \, \, p_{\ell,1}^{N\vect{d}} \, \, p_{1, k}^{N\vect{d}} \, \, X^{-1}.
\]
Then, conjugating by $(p_{k, \ell}^{N\vect{d}})^{-1}X^{-1}$, we see that $ W_{\vect{d} N} \, [a_i, a_j] \, W_{\vect{d} N}^{-1}$ is conjugate in $\gh$ to $[a_i, a_j] C_\ell^{\vect{d}}$, where $C_\ell^{\vect{d}} = \, p_{\ell,1}^{N\vect{d}} \, p_{1, k}^{N\vect{d}} \, p_{k, \ell}^{N\vect{d}}$. Since  $[a_i,a_j] C_\ell^{\vect{d}} \in \bigl\langle R'_1, R_2\bigr\rangle$, we conclude that $ W_{\vect{d} N} \, [a_i, a_j] \, W_{\vect{d} N}^{-1}\in \bigl\langle\bigl\langle R'_1, R_2\bigr\rangle\bigr\rangle$.


Now, we consider the general case, i.e., $W_{\vect{M}}\, [a_i, a_j]\, W_{\vect{M}}^{-1}$ for an arbitrary element $\vect{M} \in \ZZ^k$. Recall that $\vect{M} = \vect{d}N + \vect{s}$, where $s_i < N$ for each $i = 1, \ldots, k$, and so $\widetilde{\vect{s}} < N$.

We have the following:
$$
\begin{array}{l}
\vspace{0.2cm}
   W_{\vect{d}N + \vect{s}} \, \, [a_i, a_j]\, \, (W_{\vect{d}N + \vect{s}})^{-1}  \\
   \vspace{0.2cm}

   = a_1^{R_1(\vect{d}N + \vect{s})} \cdots a_k^{R_k(\vect{d}N + \vect{s})} \, [a_i, a_j]\,\,  (W_{\vect{d}N + \vect{s}})^{-1} \\
   \vspace{0.2cm}

   = a_1^{R_1(\vect{s})} \, \, a_1^{NR_1(\vect{d})} \, \, \cdots a_k^{R_k(\vect{s})} \, \, a_k^{NR_k(\vect{d})} \, \, [a_i, a_j] \, \, (W_{\vect{d}N + \vect{s}})^{-1} \\
   \vspace{0.2cm}
   
   =a_1^{R_1(\vect{s})} \cdots a_k^{R_k(\vect{s})} \, \, \bigl[a_k^{-R_k(\vect{s})} \cdots a_2^{-R_2(\vect{s})} \, \, \bigl(a_1^{NR_1(\vect{d})} \, \, a_2^{-NR_1(\vect{d})\vect{n_1}/\vect{n_2}}\bigr) \, \, a_2^{R_2(\vect{s})} \cdots a_k^{R_k(\vect{s})}\bigr] \\
   \vspace{0.2cm}
   
   \bigl[a_k^{-R_k(\vect{s})} \cdots a_3^{-R_3(\vect{s})} \, \, \bigl(a_2^{N(R_1(\vect{d})\vect{n_1} + R_2(\vect{d})\vect{n_2})/\vect{n_2}} \, \, a_3^{-N(R_1(\vect{d})\vect{n_1} + R_2(\vect{d})\vect{n_2})/\vect{n_3}}\bigr) \, \, a_3^{R_3(\vect{s})} \cdots a_k^{R_k(\vect{s})}\bigr]
   \cdots \\
   \vspace{0.2cm}

   \cdots \bigl[a_k^{-R_k(\vect{s})} \, \, \bigl(a_{k -1}^{N(\vect{d} - R_k(\vect{d})\vect{n_k})/\vect{n_{k -1}}} \, \, a_k^{-N(\vect{d} - R_k(\vect{d})\vect{n_k})/\vect{n_k}}\bigr) \, \, a_k^{R_k(\vect{s})}\bigr] \, \, \bigl(a_k^{N\vect{d}/\vect{n_k}} \, \, a_\ell^{-N\vect{d}/\vect{n_{\ell}}}\bigr) \\
   \vspace{0.2cm}
   
   [a_i, a_j] \, \,a_\ell^{N\vect{d}/\vect{n_{\ell}}} \, \,(W_{\vect{d}N + \vect{s}})^{-1} 
   \\
   \vspace{0.2cm}

   = W_{\vect{s}} \, \, \bigl[a_k^{-R_k(\vect{s})} \cdots a_2^{-R_2(\vect{s})} \, \, p_{1, 2}^{NR_1(\vect{d})\vect{n_1}} \, \, a_2^{R_2(\vect{s})} \cdots a_k^{R_k(\vect{s})}\bigr] \cdots \bigl[a_k^{-R_k(\vect{s})} \, \, p_{{k -1}, k}^{N(\vect{d} - R_k(\vect{d})\vect{n_k})} \, \, a_k^{R_k(\vect{s})}\bigr] \, \, p_{k, \ell}^{N\vect{d}} \\

   [a_i, a_j] \, \,  a_\ell^{N\vect{d}/\vect{n_\ell}} \, \, (W_{\vect{d}N + \vect{s}})^{-1},
\end{array}
$$
where $p_{u,v}$ are paths in $\Gamma_\chi$, $a_\ell\in V(\Gamma_\chi)$ and $(a_\ell, a_i), (a_\ell, a_j)\in E(\Gamma)$.

Let $X = [a_k^{-R_k(\vect{s})} \cdots a_2^{-R_2(\vect{s})} \, \, p_{1, 2}^{NR_1(\vect{d})\vect{n_1}} \, \, a_2^{R_2(\vect{s})} \cdots a_k^{R_k(\vect{s})}] \cdots [a_k^{-R_k(\vect{s})} \, \, p_{{k -1}, k}^{N(\vect{d} - R_k(\vect{d})\vect{n_k})} \, \, a_k^{R_k(\vect{s})}]$.

Then, 
\begin{equation}\label{equ: rel_gen}
W_{\vect{d}N + \vect{s}} \, \, [a_i, a_j]\, \, (W_{\vect{d}N + \vect{s}})^{-1} = W_{\vect{s}} \, \, X \, \, p_{k, \ell}^{N \vect{d}} \, \, [a_i, a_j] \, \, p_{\ell, 1} ^{N\vect{d}} \, \, p_{1, k}^{N\vect{d}} \, \, X^{-1} \, \, W_{\vect{s}}^{-1}.
\end{equation}

Conjugating \eqref{equ: rel_gen} by $W_{\vect{s}} X^{-1} W_{\vect{s}}^{-1}$, we obtain 
\begin{gather}\notag
\begin{split}
 W_{\vect{d}N + \vect{s}} \, \, [a_i, a_j] \, \, ( W_{\vect{d}N + \vect{s}})^{-1} & \sim W_{\vect{s}} \, \, p_{k, \ell}^{N \vect{d}} \, \, [a_i, a_j] \, \, p_{\ell, 1} ^{N\vect{d}} \, \, p_{1, k}^{N\vect{d}} \, \,  W_{\vect{s}}^{-1} \\
&= (W_{\vect{s}} \, \, p_{k, \ell}^{N \vect{d}} \, \, W_{\vect{s}}^{-1}) \, \, W_{\vect{s}} \, \, [a_i, a_j] \, \, p_{\ell, 1} ^{N\vect{d}} \, \, p_{1, k}^{N\vect{d}} \, \, W_{\vect{s}}^{-1}
\end{split}
\end{gather}
in $\gh$ (note that $W_{\vect{s}} X^{-1} W_{\vect{s}}^{-1} \in \gh$). Since $W_{\vect{s}} \, \, (p_{k, \ell}^{N \vect{d}})^{-1} \, \, W_{\vect{s}}^{-1} \in \gh$ again conjugating by $W_{\vect{s}} \, \, (p_{k, \ell}^{N \vect{d}})^{-1} \, \, W_{\vect{s}}^{-1}$, we see that 
$$
 W_{\vect{d}N + \vect{s}} \, \, [a_i, a_j]\, \, (W_{\vect{d}N + \vect{s}})^{-1} 
 \sim 
 W_{\vect{s}} \, \, [a_i, a_j]  \, \, p_{\ell,1}^{N\vect{d}} \, \, p_{1, k}^{N\vect{d}} \, \, p_{k, \ell}^{N\vect{d}} \, \,W_{\vect{s}}^{-1} 
 =
 ( W_{\vect{s}}\, \, [a_i, a_j] \, \, W_{\vect{s}}^{-1}) \, \,( W_{\vect{s}} \, \, C_\ell^{\vect{d}}\, \, W_{\vect{s}}^{-1}),
$$
where $C_\ell^{\vect{d}}= \, p_{\ell,1}^{N\vect{d}} \, \, p_{1, k}^{N\vect{d}} \, \, p_{k, \ell}^{N\vect{d}}$.  Therefore, $W_{\vect{M}}[a_i, a_j]W_{\vect{M}}^{-1} \in \bigl\langle\bigl\langle R'_1, R_2 \bigr\rangle\bigr\rangle$.

Straightforward computation shows that cycles define the trivial element in $\gh$ and so $R'_1,R_2, R_3 \in \bigl\langle\bigl\langle R_3,\, W_{\vect{M}}[a_i, a_j]W_{\vect{M}}^{-1} : \vect{M} \in \ZZ^k\bigr\rangle\bigr\rangle$. 


The only infinite family of relations in the presentation is the family $R'_1$. We record the following, which is analogous~\cite[Proposition 2]{DL}. 

 Recall that $\triangle_\Gamma$ is $(k-1)$-$1$-connected, and so it is simply connected. Also, if $\triangle_\Gamma$ is simply connected, cycles of length $3$ in $\Gamma$ generate the fundamental group of the graph $\Gamma$, and, in this case, Dicks--Leary~\cite{DL} proved that the relations for $\ell$-cycles are a consequence of the relations for $3$-cycles (together with $R_2, R_3$).

 Let us suppose that $(a_{i_1}, a_{i_2}, a_{i_3})$ be any directed $3$-cycle in $\triangle_\Gamma$ and denote $e_1 = (a_{i_1}, a_{i_2}),\, e_2 = (a_{i_2}, a_{i_3})$ and $e_3 = (a_{i_3}, a_{i_1})$.
    The relation $e^N_1 e^N_2 e^N_3 = 1$
    implies that $e^N_3 = e^{-N}_2 e^{-N}_1$. Since for any vectorised weighted edge $e^{-N} = (e^N)^{-1}$, the relation $e^{-N}_1 e^{-N}_2 e^{-N}_3 = 1$
    %
 is equivalent to $ e^{-N}_1 e^{-N}_2 e^N_1 e^N_2 = 1$. Also, for any vectorised weighted edge, $e^{N\vect{d}}_{j,k} = (a^{[\vect{N}]/\vect{n_j}}_j \, a^{-[\vect{N}]/\vect{n_k}}_k)^{\vect{d}} =(a^{\frac{[\vect{N}]\vect{d}}{\vect{n_j}}}_j \, a^{-\frac{[\vect{N}]\vect{d}}{\vect{n_k}}}_k)$ holds.
Thus, it follows that $\tau(W_{\vect{s}} e_1^{N\vect{d}} \, e_2^{N\vect{d}} \, e_3^{N\vect{d}} W_{\vect{s}}^{-1}) = 1$ is a consequence of the relations $\tau(W_{\vect{s}} e_1^{N} \, e_2^{N} \, e_3^{N} W_{\vect{s}}^{-1})$ and $\tau(W_{\vect{s}} e_1^{-N} \, e_2^{-N} \, e_3^{-N} W_{\vect{s}}^{-1})$ together with $R_2$ and $R_3$. The reader is referred to~\cite[Proposition 2]{DL} for details.
Hence the announced presentation follows from Lemma \ref{lem:trasverse_homomorphism}.

We are left to prove that if $\gh$ is finitely presented then $\triangle_\Gamma$ is $(k-1)$-$1$-connected and we prove it by contradiction. 

Let us assume that $H_\psi$ is finitely presented.
If $\triangle_\Gamma$ is not $(k-1)$-$1$-connected then there exists a set of vertices $\{ a_{i_1}, \ldots,  a_{i_d}\}$ with $0 \leqslant d \leqslant k-1$ and $d < |V(\Gamma)|$ such that $\triangle_\Gamma \setminus \{ a_{i_1}, \ldots,  a_{i_d}\}$ is not $1$-connected. Without the loss of generality we may assume that $\{ a_{1}, \ldots,  a_{d}\}$ be that set. Recall that $\psi(a_i) = \vect{n_i}$ for each $i= 1, \ldots, s$. Since, $\psi$ is surjective, we can choose a non-zero element $\zeta \in \Hom(\ZZ^k, \ZZ)$ which maps only $\{ \vect{n_1}, \ldots, \vect{n_d}\}$ to $0$. Then, clearly $\triangle_{\zeta \circ \psi} = \triangle_\Gamma \setminus \{ a_{1}, \ldots,  a_{d}\}$ is not simply connected. Then from Proposition~\ref{prop 1}, $H_\psi$ is not finitely presented and which is a contradiction. This completes the proof.
\end{proof}



\section{Presentation of the kernels of the characters}

In this section, we employ Theorem~\ref{thm:finite_presentation} to deduce the finite presentation of the kernels of characters of a RAAG $G_\Gamma$, say $\chi \colon G_\Gamma \to \mathbb R$. And this enables us to deduce the finite presentation of the full normal subgroup of any RAAG $G_\Gamma$. 
At the end, we also recover the finite presentation of rational characters of RAAGs (see Theorem~\ref{cor: CKRoy thm}) and the finite presentation of so-called Bestvina--Brady groups groups (see Corollary~\ref{cor: DL thm})

\begin{thm}\label{thm: FP_n_implies_semidirect}
Let $\gh$ be the kernel of an epimorphism $\psi: G_\Gamma \to \mathbb Z^k$. If $\gh$ is of type $FP_k$, then $G_\Gamma \simeq \gh \rtimes \ZZ^k$.
\end{thm}

\begin{proof}

Let $s=|V(\Gamma)|$. Notice that since $\psi$ is an epimorphism, we have that $s\geqslant k$. We proceed by induction on $\ell = s-k$.

If $\ell=0$, i.e. $s= k$, then $\ZZ^k$ is (isomorphic to) the abelianisation of $G_\Gamma$ and $\gh$ is the commutator subgroup of $G_\Gamma$. Then $\gh$ is finitely generated if and only if it is of type $FP_k$ if and only if $G_\Gamma \simeq \ZZ^k$ and so the result holds trivially.

Let $\ell >0$. If $G_\Gamma$ is free abelian, then the result holds.

Suppose that $G_\Gamma$ is not free abelian. Then there exists $v\in v(\Gamma)$ such that $\st(v) \ne \Gamma$ and we have a nontrivial decomposition $G_\Gamma \simeq G_{\Gamma\setminus \{v\}} \ast_{G_{\link(v)}} G_{\st(v)}$. 

We observe that if $\gh$ is of type $FP_k$, then $\psi(G_{\link(v)})$ is free abelian of rank $k$. Indeed, otherwise there is an epimorphism $\zeta: \ZZ^k \to \ZZ$ such that $\zeta(\psi(G_{\link(v)}))=1$, i.e. $\psi(G_{\link(v)})$ is in the kernel of the character $\chi= \zeta \circ \psi$, $\chi: G_\Gamma = G_\Gamma \simeq G_{\Gamma\setminus \{v\}} \ast_{G_{\link(v)}} G_{\st(v)} \to \ZZ$. From \cite[Corollary 2.6]{CL}, it follows that the kernel of $\chi$ is not finitely generated and so $\chi \notin \Sigma^1(G_\Gamma, R)$. 
From the definition of $\chi$, we have that $\chi(\gh)=1$ and so $\chi \in S(G_\Gamma, \gh)$. From \cite[Theorem 2.1]{MMV}, this contradicts that $\gh$ is of type $FP_k$. Therefore, we have that $\psi(G_{\link(v)})\simeq \ZZ^k$.

Consider the (restriction) epimorphism $\psi'=\psi|_{G_{\link(v)}}:G_{\link(v)} \to \ZZ^k$ with kernel $H_{\link(v)}$. Since $G_{\link(v)}$ is a retraction of $G_\Gamma$ and $\psi'$ is defined as the restriction of $\psi$ on $G_{\link(v)}$, we have that $H_{\link(v)}$ is a retraction of $\gh$ and since $\gh$ is of type $FP_k$ so is $H_{\link(v)}$. 

Since $|V(\link(v))| - k < |V(\Gamma)| - k$, by induction, we have that $G_{\link(v)} \simeq H_{\link(v)} \rtimes \ZZ^k$. In particular, $G_{\link(v)}$ retracts to $\ZZ^k$ and so does $G_\Gamma$. We conclude that $G_\Gamma \simeq \gh \rtimes \ZZ^k$.
\end{proof}




\begin{cor}\label{cor:FP_n_implies_semidirect}
Let $\gh$ be the kernel of an epimorphism $\psi: G_\Gamma \to \ZZ^n$. If $\gh$ is of type $FP_n$, then $G_\Gamma \simeq \gh \rtimes \mathbb Z^n$. In particular $\gh$ admits a presentation as in Theorem \ref{thm:finite_presentation}.
\end{cor}

\begin{proof}
    The proof follows directly from Theorem~\ref{thm: FP_n_implies_semidirect}.
\end{proof}

For a character $\chi: G_\Gamma \to \mathbb R$ such that $\chi(G_\Gamma) \simeq Q \times \ZZ^k$ where $Q$ is a finite group and $k \leqslant |V(\Gamma)|$, we get the following result. 

\begin{cor}\label{cor:character}
Let $G_\Gamma$ be a RAAG with $|V(\Gamma)|=m$ and $H_{\chi}$ be the kernel of a character $\chi$ as defined above. Also let $\triangle_\Gamma$ is $(k-1)$-$1$-connected. 

If $H_{\chi}$ is of type $FP_m$, then one can give an explicit presentation for $H_{\chi}$.
\end{cor}

\begin{proof}
The image $\chi(G_\Gamma)$ is isomorphic to $Q \times \ZZ^k$ where $Q$ is a finite group and $k \leqslant m$. Furthermore, this isomorphism can be described algorithmically. 
Let $\psi: G_\Gamma \to \ZZ^k$ and let $\gh$ be the kernel. We have that $H_\chi < \gh$ is a finite index subgroup of $\gh$ (in fact, of index $|Q|$) and so $\gh$ is of type $FP_m$, see \cite[Proposition 2.2]{KF}. From Corollary \ref{cor:FP_n_implies_semidirect}, we have that $\gh$ admits a presentation as in Theorem \ref{thm:finite_presentation}. Using the Reidemeister-Schreier Rewriting Process, one can give an explicit presentation for the subgroup $H_\chi$. 
\end{proof}

Any RAAG $G_\Gamma$ admits a decomposition $G_{\Gamma_1} \times G_{\Gamma_2} \times \cdots \times G_{\Gamma_n}$ as a direct product where $G_{\Gamma_i}$ is a directly indecomposable RAAG. In general, A subgroup $H \leqslant  G_1 \times \cdots \times G_n$ is called \emph{full} if $H$ intersects nontrivially
each factor, i.e., $H \cap G_i \neq 1$ for all $i = 1, \dots, n$. Notice that if $H$ is full, in particular it is nontrivial and if $H$ intersects trivially one of the factors, say $H \cap G_1 = 1$, then $H$ is (isomorphic to) a subgroup of $G_2 \times \cdots \times Gn$.

In \cite{CL}, the authors show that any finitely generated full normal subgroup $N$ of $G_\Gamma$ is a finite index subgroup of some $H_\Gamma$ and so, combining these results, one obtains presentations for normal subgroups $N$ of type $FP_m$ where $m =|V(\Gamma)|$.

\begin{cor}\label{cor:finiteness_prop}
Let $G_\Gamma$ be a RAAG where $\triangle_\Gamma$ is simply connected and let $m=|V(\Gamma)|$. Let $N < G_\Gamma$ be a finitely generated full normal subgroup of type $FP_m$. Then one can give an explicit presentation for $N$.   
\end{cor}


\begin{proof}
From \cite[Corollary 2.6]{CL}, $N$ is a finite index subgroup of a kernel of a character $H_\chi$. From Corollary \ref{cor:character}, we can algorithmically determined the presentation of $H_\chi$ and so using the Reidemeister-Schreier Rewriting Process, one can determine the presentation of $N$.   
\end{proof}

From the finite presentation of~\ref{thm:finite_presentation}, one can also get the finite presentation of the kernels of the rational characters of RAAGs by choosing appropriate set of right coset representatives, see~\cite[Theorem 3.16]{CKRoy} for details. In particular, we can also recover Dicks and Leary's presentation for Bestvina--Brady groups see~\cite[Corollary 3.17]{CKRoy}.

Notice that any rational character $\varphi \colon G_\Gamma \to \ZZ$ can be viewed as $  \zeta \circ \psi$, where $\psi \colon G_\Gamma \to \ZZ^k$ and $\zeta \colon \ZZ^k \to \ZZ$. The flag complex $\triangle_\Gamma$ is $(k-1)$-$1$-connected implies that $\triangle_{\zeta \circ \psi}$, i.e., $\triangle_\varphi$ is simply connected and $1$-acyclic-dominating from Fact~\ref{fact 1}. Let us suppose that $\varphi \colon G_\Gamma \to \ZZ$ maps $a_i$ to $n_i$ for $i = 1, \dots, s$. Since $\varphi$ is an epimorphism there exists $w \in G_\Gamma$ such that $\varphi(w) =1$. Now, following the notation of~\cite{CKRoy} if we choose $\{ w_m \mid m \in \ZZ\}$, where $w_m = a^{mr_1}_1 \cdots a^{mr_\ell}_\ell$ as the right coset representatives and $x_{m,i} = w_m a_i (w_{m+n_i})^{-1}$, then comparing Theorem~\ref{thm:finite_presentation} and \cite[Theorem 3.16]{CKRoy} we have $k = \ell$, $R = R_{\mathcal{M}} = r = \max \{ |r_1|, \dots, |r_\ell|\}$, and finally, $\widetilde{\vect{m}}$ will coincide with $|m|$. Hence, we get the following:

\begin{thm}[cf. Theorem 3.16 of \cite{CKRoy}]\label{cor: CKRoy thm}
    Let $\triangle_\varphi$ be simply connected and $1$-acyclic-dominating. Then the kernel $\ker \varphi=\gh$ has the following finite presentation 
    $$
    \bigl\langle x_{m,i} \, \, \mid \, \, R_1', R_2, R_3\bigr\rangle, 
    $$
    where $x_{d,i}=w_d a_i w_{d+n_i}^{-1}$, $0 \leqslant |m| < \ell r N^2$. The relations $R_1, R_2, R_3$ are defined as follows:  
\begin{itemize}
\item[$R_1$:] for any directed $3$-cycle $(a_{i_1}, a_{i_2}, a_{i_3})$,
\[
 \tau(w_s
 e_{i_1,i_2}^{Nq} e_{i_2,i_3}^{Nq} e_{i_{3},i_1}^{Nq}  w_s^{-1}), \hspace{0.2cm} 0 \leqslant |s| < N, q=\pm 1;
\]
\item[$R_2$:] $\tau(w_s [a_i, a_j] w_s^{-1})=1, \hspace{0.2cm} 0 \leqslant |s| < N$, where $(a_i, a_j) \in E(\Gamma);$
\item [$R_3$:] $x_{t,i}=\tau(w_t a_i (w_{t+n_i})^{-1})$ for $0\leqslant |t| < \ell^2r^2N^3$.
\end{itemize}
\end{thm}

Similarly, in the case when $n_i=1$ for all $i$, that is when the kernel is a Bestvina--Brady group, we have that $\triangle_\Gamma=\triangle_\varphi$ and so $\triangle_\varphi$ being simply connected and $1$-acyclic-dominating is equivalent to $\triangle_\Gamma$ being simply connected. In this case, we also have that $\ell=r=N=1$ and, without the loss of generality, we can take $w=a_1$. We recover Dicks and Leary's presentation:

\begin{cor}[cf. Proposition 2 and Corollary 3 of \cite{DL}, Corollary 3.17 of \cite{CKRoy}]\label{cor: DL thm}
Let $\triangle$ be simply connected and let $H_\Gamma$ be the Bestvina--Brady subgroup of $G_\Gamma$. Then $H_\Gamma$ admits the finite presentation
$$
    \bigl\langle x_{0,i} \, \, \mid \, \, \tau(e_{i_1,i_2}^{q} e_{i_2,i_3}^{q} e_{i_3,i_1}^{q}) \bigr\rangle, 
    $$
    where $x_{0,i}= a_ia_1^{- 1}$, $a_i\in V(\Gamma)$, $e_{i_r,i_s}=a_{i_r}a_{i_s}^{-1}$ and $e_{i_1,i_2} e_{i_2,i_3} e_{i_3,i_1}$ is a directed $3$-cycle, and $q=\{\pm 1\}$.
\end{cor}

\bigskip

\noindent\textbf{Acknowledgments:} The author is grateful to Montserrat Casals-Ruiz for her insightful suggestions on this work and discussions happened during the author's stay in the University of the Basque Country, Spain.


\end{document}